\documentclass[11pt]{amsart}
\usepackage[utf8]{inputenc}

\usepackage[plainpages=false]{hyperref}
\usepackage{amsfonts,latexsym,rawfonts,amsmath,amssymb,amsthm, mathrsfs, lscape}
\usepackage{verbatim}

\usepackage[all]{xy}
\usepackage{graphicx,psfrag}

\usepackage{array,tabularx}

\usepackage{anysize}

\newtheorem{thm}{Theorem}
\newtheorem{cor}{Corollary}[section]

\newtheorem{prop}{Proposition}[section]

\theoremstyle{remark}
\newtheorem{rmk}{Remark}[section]

\theoremstyle{definition}
\newtheorem{defn}{Definition}[section]

\newtheorem*{theorem*}{Theorem}

\newtheorem{theorem}{Theorem}[section]

\newtheorem{lemma}[theorem]{Lemma}

\marginsize{3cm}{3cm}{2cm}{3cm}
\def\p{\partial}
\def\R{\mathbb{R}}
\def\C{\mathbb{C}}

\def\l{\lambda}

\def\i{\sqrt{-1}}

\def\o{\omega}

\def\cD{\mathcal D}
\def\cE{{\mathcal E}}
\def\cF{{\mathcal F}}

\def\cH{{\mathcal H}}

\def\cK{{\mathcal K}}
\def\cL{{\mathcal L}}

\def\cO{{\mathcal O}}

\def \bp {\overline{\partial}}

\def \Ker {\text{Ker}}

\begin{document}

\title{Scalar curvature and properness on Sasaki manifolds}
\author{Weiyong He}
\address{Department of Mathematics, University of Oregon, Eugene, OR 97403. }
\email{whe@uoregon.edu}
\begin{abstract}
We study (transverse) scalar curvature type equation on compact Sasaki manifolds, in view of recent breakthrough of Chen-Cheng \cite{CC1, CC2, CC3} on existence of K\"ahler metrics with constant scalar curvature (csck) on compact K\"ahler manifolds. Following their strategy, we prove that given a Sasaki structure (with Reeb vector field and complex structure on its cone fixed ), there exists a Sasaki structure with transverse constant scalar curvature (cscs) if and only if the $\cK$-energy is reduced proper modulo the identity component of the automorphism group which preserves both the Reeb vector field and transverse complex structure. Technically, the proof mainly consists of two parts. The first part is a priori estimates for scalar curvature type equations which are parallel to Chen-Cheng's results in  \cite{CC2, CC3} in Sasaki setting. The second part is geometric pluripotential theory on a compact Sasaki manifold, building up on profound results in geometric pluripotential theory on K\"ahler manifolds. There are notable, and indeed subtle differences in Sasaki setting (compared with K\"ahler setting) for both parts (PDE and pluripotential theory). The PDE part is an adaption of deep work of Chen-Cheng \cite{CC1, CC2, CC3} to Sasaki setting with necessary modifications, where Chen's continuity path plays a very important role.
While the geometric pluripotential theory on a compact Sasaki manifold has new  difficulties, compared with  geometric pluripotential theory in K\"ahler setting which is very intricate. We shall present the details of geometric pluripotential on Sasaki manifolds in a separate paper \cite{HL} (joint work with Jun Li). 
\end{abstract}

\maketitle

\section{Introduction}
Sasaki manifolds have gained their prominence in physics and in algebraic geometry \cite{BG}. Technically Sasaki geometry is an odd dimensional analogue of K\"ahler geometry and almost all results in K\"ahler geometry have their counterparts in Sasaki geometry.  Calabi's extremal metric \cite{Ca1, Ca2} (and csck) has played a very important role in K\"ahler geometry. In 1997, S. K. Donaldson \cite{Don97} proposed an extremely fruitful program to approach existence of csck (extremal metrics) on a compact K\"ahler manifold with a fixed K\"ahler class. Tremendous work and progress have been made to characterize exactly when a K\"ahler class contains a csck (extremal K\"ahler metric). The analytic part for existence of csck is to solve a fourth order highly nonlinear elliptic equation, the scalar curvature type equation. This problem is regarded as a very hard problem in the field. Recently Chen and Cheng \cite{CC1, CC2, CC3} have solved  a major conjecture that existence of csck is equivalent to well studied conditions such as properness of Mabuchi's $K$-energy, or geodesic stability. In this paper we  prove the following result for Sasaki manifolds. 

\begin{thm}\label{cscs}There exists a Sasaki metric with constant scalar curvature if and only if the $\cK$-energy is reduced proper with respect to $\text{Aut}_0(\xi, J)$, the identity component of automorphism group which preserves the Reeb vector field and transverse complex structure. 
\end{thm}

We consider Sasaki structures induced by transverse K\"ahler potentials, with the Reeb vector field and the transverse complex structure both fixed. The reduced properness seems to be precise notion of properness for Mabuchi's $\cK$-energy. Roughly speaking, we test the properness of the $\cK$-energy against the Finsler distance $d_1$ on the space of transverse K\"ahler potentials, modulo the action of $\text{Aut}_0(\xi, J)$. We shall be precise about the statement of Theorem \ref{cscs} in Section 2. Theorem \ref{cscs} also holds for Sasaki-extremal metric, with the $\cK$-energy replaced by modified $\cK$-energy, with necessary modifications in line with \cite{He18}. The details will appear in \cite{HL}. 

Our proof of Theorem \ref{cscs} follows the same strategy as in K\"ahler setting \cite{CC3}.  
We first adapt Chen-Cheng's method to prove a priori estimates in Sasaki setting. To formulate these results precisely, we recall some notations in Sasaki geometry. 
 Let $(M, g)$ be a compact Riemannian manifold of dimension $2n+1$, with a Riemannian metric $g$. Sasaki manifolds have very rich geometric structures and have many equivalent descriptions. A probably most straightforward formulation is as follows: its metric cone \[X=\R_{+}\times M,  g_X=dr^2+r^2 g.\] is a K\"ahler cone i.e. there exists a complex structure $J$ on $X$ such that $(g_X, J)$ defines a K\"ahler structure.  We identify $M$ with its natural embedding $M\rightarrow \{r=1\}\subset X$. 
 The $1$-form $\eta$ is given by $\eta=J(r^{-1}dr)$ and it defines a contact structure on $M$. The vector field $\xi:=J(r\p_r)$ is a nowhere vanishing, holomorphic Killing vector field and it is called the \emph{Reeb vector field} when it is restricted on $M$. The integral curves of $\xi$ are geodesics, and give rise  to a foliation structure on $M$, called the \emph{Reeb foliation}. Then there is a K\"ahler structure on the local leaf space of the Reeb foliations, called the \emph{transverse K\"ahler structure}.  A standard example of a Sasaki manifold is the odd dimensional round sphere $S^{2n+1}$. The corresponding K\"ahler cone is $\C^{n+1}\backslash\{0\}$ with the flat metric and its transverse K\"ahler structure descends to $\mathbb{CP}^2$ with the Fubini-Study metric. 
 
We can also formulate Sasaki geometry, in particular its transverse K\"ahler structure via its contact bundle $\cD=\text{Ker}(\eta)\subset TM$. The complex structure $J$ on the cone descends to the contact bundle via $\Phi:=J|_\cD$. The Sasaki metric can be written as follows,
 \begin{equation*}
 g=\eta\otimes\eta+g^T,
 \end{equation*}
 where $g^T$ is the transverse K\"ahler metric, given by $g^T:=2^{-1}d\eta(\Phi\otimes \mathbb{I})$. The transverse K\"ahler form is denoted by $\omega^T=2^{-1}d\eta$. 
We want to study the transverse K\"ahler geometry of Sasaki metrics, with the Reeb vector field $\xi$ and transverse complex structure (equivalently the complex structure $J$ on the cone) both fixed.
This means that we fix the basic K\"ahler class $[\omega^T]$ with $\omega^T=2^{-1}d\eta$ and study the Sasaki structures induced by the space of transverse K\"ahler potentials,
\begin{equation}
\cH=\{\phi\in C_B^\infty(M): \omega_\phi=\omega^T+\p_B\bar\p_B \phi>0\}
\end{equation}
where $C_B^\infty$ stands for smooth \emph{basic functions} $d\phi(\xi)=0$. 

To prove Theorem \ref{cscs}, Chen's continuity path \cite{chen15} plays a very important tole, as in K\"ahler setting. 
Given all analytic results including a priori estimates for scalar curvature type equation and profound results in pluripotential theory, this continuity path seems still to be inevitable as an instrumental tool to prove Theorem \ref{cscs}. 
We study Chen's continuity path in Sasaki setting,
\begin{equation}
t(R^T_\phi-\underline R)=(1-t) (\text{tr}_{\omega_\phi} \omega^T-n).
\end{equation}
or more generally, for a given real basic $(1, 1)$ form $\beta$ and a basic function $h$. 
\begin{equation}\label{scalar01}
R^T_\phi=\text{tr}_\phi \beta+h.
\end{equation}
We introduce another notation, the log-volume ratio $F:=F_\phi$,
\begin{equation*}
\eta\wedge \omega_\phi^n=\eta_\phi\wedge (\omega_\phi)^n=e^F\eta\wedge (\omega^T)^n
\end{equation*}
Denote $dv_g=\eta\wedge \frac{(\omega^T)^n}{n!}$ and $dv_\phi=\eta_\phi\wedge \frac{(\omega_\phi)^n}{n!}$. We introduce the entropy $H(\phi)$,
\begin{equation}\label{entropy101}
H(\phi):=\int_M \log \frac{\omega_\phi^n\wedge \eta}{\omega_T^n\wedge \eta}dv_\phi=\int_M e^F F dv_g
\end{equation}
We have the following a priori estimates, 
\begin{thm}\label{main00}
Let $(M, g, \xi, \eta)$ be a compact Sasaki manifold. Consider a basic function $\phi\in \cH, \sup \phi=0$ satisfying \eqref{scalar01}. Then there exists a constant $C_0>1$ such that
\begin{equation}
C_0^{-1}\leq n+\Delta \phi\leq C_0, \|\phi\|_{C^k}\leq C=C(k, C_0),
\end{equation}
where $C_0$ is positive bounded constant depending on the background metric $(M, g)$, the upper bound of $H(\phi)$, $\max|\beta|_g$ and $\max|h|$. 
\end{thm}

It is important to consider the scalar curvature type equation in a more general form,
\begin{equation}\label{scalar02}
R^T_\phi=\text{tr}_\phi \beta+h,\end{equation}
where $ \beta:=\beta_0+\sqrt{-1}\p\bar\p f\geq 0$ and $f$ is a smooth basic function with $\sup f=0$ and \[\int_M e^{-p_0 f}dv_g<\infty.\]
We assume $p_0$ is a sufficiently large constant depending only on $n$ ($p_0=100(n^2+1)$ is sufficient for example). 
Then we have the following
\begin{thm}\label{main01}
Let $(M, g, \xi, \eta)$ be a compact Sasaki manifold. Consider a basic function $\phi\in \cH, \sup \phi=0$ satisfying \eqref{scalar02}. Then there exists a constant $C_0>1$ such that, for $1\leq p\leq p_0-1$,
\begin{equation}
\|F+f\|_{W^{1, 2p}}+\|n+\Delta \phi\|_{L^p}\leq C_0
\end{equation}
where $C_0$ is a positive bounded constant depending on the background metric $(M, g)$, the upper bound of $H(\phi)$, $\max|\beta_0|_g$, $\max|h|$ and $\int_M e^{-p_0 f}dv_g$. 
\end{thm}

\begin{rmk}Theorem \ref{main00} and Theorem \ref{main01} are direct adaption of results in \cite{CC2, CC3} to the Sasaki setting.  We follow the method in \cite{CC3}[Section 2], including our proof of Theorem \ref{main00}, which corresponds to \cite{CC2}[Theorem 3.1]. The method in \cite{CC3} simplifies the proof considerately. The integral method as in \cite{chenhe12} has played an important role in these estimates. 
\end{rmk}

With these estimates, another central piece for Theorem \ref{cscs} is the pluripotential theory on Sasaki manifolds; for us the most relevant results would be in \cite{GZ, BB, D2} and \cite{D4} gives a very nice reference. 
One of the key results is the following geometric structure of the metric completion of $\cH$, denoted by $(\overline{\cH}, d_1)=(\cE_1, d_1)$ with the Finsler metric $d_1$. The Orlicz-Finsler geometric on $\cH$ was first introduced by T. Darvas \cite{D1, D2} in K\"ahler setting and it has played an important role in the proof of properness conjecture \cite{BDL2, CC2, CC3}. We have the following, 

\begin{thm}[He-Li \cite{HL}]\label{pluri01}$(\cE_1(M, \xi, \omega^T), d_1)$ is a geodesic metric space, which is the metric completion of $(\cH, d_1)$. For any $u, v\in \cE_1$, there exists  a uniform constant $C=C(n)>1$ such that
\begin{equation}\label{d01}
C^{-1} I_1(u, v)\leq d_1(u, v)\leq CI_1(u, v),
\end{equation}
where the energy functional $I_1$ is given by
\[
I_1(u, v)=\int_M |u-v| \omega_u^n\wedge \eta+\int_M |u-v|\omega_v^n\wedge \eta. 
\] 
Moreover, we have
\begin{equation}\label{d02}
d_1(u, \frac{u+v}{2})\leq C d_1(u, v). 
\end{equation}
\end{thm}

This is a counterpart of T. Darvas's results \cite{D1, D2} in K\"ahler setting. We shall prove this theorem following the lines closely as in \cite{D4}. 
Given this result, one can then extend the $\cK$-energy to $\cE_1(M, \xi, \omega^T)$, and keep it still convex along geodesics, see \cite{BB, JZ, VC, BDL2}.
We actually generalize almost all related results in \cite{D4} to Sasaki setting, building up on profound results in pluripotential theory by many and geodesic equation \cite{GZ1}. We refer readers to \cite{D4} for  references. We should emphasize that the geometric pluripotential theory on a compact Sasaki manifold does impose new difficulties, in particular when the Reeb foliation is irregular. The arguments are  tricky at times and lengthy.  We will summarize the results below for pluripotential theory on compact Sasaki manifolds. We shall present these results, including the proof of Theorem \ref{pluri01},  in a forth-coming paper \cite{HL}. 

There are tremendous work in the last two decades in Sasaki geometry, in particular on Sasaki-Einstein manifolds, see \cite{BG, GMSW, BGK, FOW, MSY, Sparks, zhang1, HeSun, CS} and reference therein. Calabi's extremal metric has a direct adaption in Sasaki setting \cite{BGS1} and Donaldson's program has also been extended to Sasaki setting, see \cite{GZ} for example. Even though we only state Theorem \ref{pluri01} for $d_1$, but it holds for a more general setting (see \cite{D4} and \cite{HL}). The Riemannian distance $d_2$ plays a prominent role, which was studied extensively in \cite{GZ1}, as a counterpart of Chen's results \cite{chen01} in K\"ahler setting. \\

{\bf Acknowledgement:} The author thanks Prof. Xiuxiong Chen sincerely for encouragement on this work and constant support. It is evident that the profound results in K\"ahler geometry make it possible for us to deal with the Sasaki case.  
The author is supported in part by an NSF grant, award no. 1611797.

\numberwithin{equation}{section}
\numberwithin{thm}{section}

\section{Preliminary on Sasaki geometry}A good reference on Sasaki geometry can be found in the monograph \cite{BG} by Boyer-Galicki. 
Let $M$ be a compact differentiable manifold of dimension $2n+1 (n\geq 1)$. A \emph{Sasaki structure} on $M$ is defined to be a K\"ahler cone structure on $X=M\times \R_{+}$, i.e. a K\"ahler metric $(g_X, J)$ on $X$  of the form $$g_X=dr^2+r^2g,$$ where $r>0$ is a coordinate on $\R_{+}$, and $g$ is a Riemannian metric on $M$. We call $(X, g_X, J)$ the \emph{K\"ahler cone} of $M$. We also identify $M$ with the link $\{r=1\}$ in $X$ if there is no ambiguity.  Because of the cone structure, the K\"ahler form on $X$ can be expressed as 
$$\omega_X=\frac{1}{2}\sqrt{-1}\p\bp r^2=\frac{1}{4}dd^c r^2.$$
We denote by $r\p_r$ the homothetic vector field on the cone, which is easily seen to be a real holomorphic vector field. 
A tensor $\alpha$ on $X$ is said to be of homothetic degree $k$ if 
$$\cL_{r\p_r} \alpha=k\alpha.$$
In particular, $\omega$ and $g$ have homothetic degree two, while $J$ and $r\p_r$ has homothetic degree zero.
We define the \emph{Reeb vector field} $$\xi=J(r\p_r).$$
Then $\xi$ is a holomorphic Killing field on $X$ with homothetic degree zero. Let $\eta$ be the dual one-form to $\xi$:
 \[\eta(\cdot)=r^{-2}g_X(\xi, \cdot)=d^c \log r=\i (\bp-\p)\log r\ .\] 
We also use $(\xi, \eta)$ to denote the restriction of them on $(M, g)$.  Then we have 
\begin{itemize}
\item $\eta$ is a contact form on $M$, and $\xi$ is a Killing vector field on $M$ which we also call the Reeb vector field;
\item $\eta(\xi)=1, \iota_{\xi} d\eta(\cdot)=d\eta (\xi, \cdot)=0$;
\item the integral curves of $\xi$ are geodesics.  
\end{itemize}

The Reeb vector field $\xi$ defines a foliation $\cF_\xi$ of $M$ by geodesics. There is a classification of Sasaki structures according to the global property of the leaves. If all the leaves are compact, then $\xi$ generates a circle action on $M$, and the Sasaki structure is called {\it quasi-regular}. In general this action is only locally free, and we get a polarized orbifold structure on the leaf space. If the circle action is globally free, then the Sasaki structure is called {\it regular}, and the leaf space is a polarized K\"ahler manifold. If $\xi$ has a non-compact leaf the Sasaki structure is called {\it irregular}. Regularity will not appear essential explicitly in the paper, but we shall see that irregular Sasaki structures do impose substantial difficulties at many occasions.

There is an orthogonal decomposition of the tangent bundle \[TM=L\xi\oplus \cD,\] where $L\xi$ is the trivial  bundle generalized by $\xi$, and $\cD=\Ker (\eta)$.
The metric $g$ and the contact form $\eta$ determine a $(1,1)$ tensor field $\Phi$ on $M$ by
\[
g(Y, Z)=\frac{1}{2} d\eta(Y, \Phi Z), Y, Z\in \Gamma(\cD). 
\]
$\Phi$ restricts to an almost complex structure on $\cD$: \[\Phi^2=-\mathbb{I}+\eta\otimes \xi. \] 

Since both $g$ and $\eta$ are invariant under $\xi$, there is a well-defined K\"ahler structure $(g^T, \omega^T, J^T)$ on the local leaf space of the Reeb foliation. We call this a \emph{transverse K\"ahler structure}. In the quasi-regular case, this is the same as the K\"ahler structure on the quotient. Clearly 
\[
\o^T=\frac{1}{2}d\eta. 
\]
The upper script $T$ is used to denote both the transverse geometric quantity, and the corresponding quantity on the bundle $\cD$. For example we have on $M$
$$g=\eta\otimes  \eta+g^T.$$
From the above discussion it is not hard to see that there is an intrinsic formulation of a Sasaki structure as a compatible integrable pair $(\eta, \Phi)$, where $\eta$ is a contact one form and $\Phi$ is a almost CR structure on $\mathcal D=\Ker \eta$. Here ``compatible" means  first that 
$d\eta(\Phi U, \Phi V)=d\eta(U, V)$ for any $U, V\in \mathcal D$, and $d\eta(U, \Phi U)>0$ for any non zero $U\in \mathcal D$. Further we require $\mathcal L_{\xi}\Phi=0$, where $\xi$ is the unique vector field with $\eta(\xi)=1$, and $d\eta(\xi, \cdot)=0$. 
$\Phi$ induces a splitting $$\mathcal  D\otimes \C=\mathcal D^{1,0}\oplus \mathcal D^{0,1}, $$
 with $\overline{\mathcal D^{1,0}}=\mathcal D^{0,1}$. 
``Integrable" means that $[\mathcal D^{0,1}, \mathcal D^{0,1}]\subset \mathcal D^{0,1}$. This is equivalent to that the induced almost complex structure on the local leaf space of the foliation by $\xi$ is integrable. For more discussions on this, see \cite{BG} Chapter 6. 

\begin{defn} A $p$-form $\theta$ on $M$ is called basic if
\[
\iota_\xi \theta=0, L_\xi \theta=0.
\]
Let $\Lambda^p_B$ be the bundle of basic $p$-forms and $\Omega^p_B=\Gamma(S, \Lambda^p_B)$ the sections of $\Lambda^p_B$.  
\end{defn}
The exterior differential preserves basic forms. We set $d_B=d|_{\Omega^p_B}$. 
Thus the subalgebra $\Omega_{B}(\cF_\xi)$ forms a subcomplex of the de Rham complex, and its cohomology ring $H^{*}_{B}(\cF_\xi)$  is called the {\it basic cohomology ring}. When $(M, \xi, \eta, g)$ is a Sasaki structure, there is a natural splitting of $\Lambda^p_B\otimes \C$ such that
\[
\Lambda^p_B\otimes \C=\oplus \Lambda^{i, j}_B,
\]
where $\Lambda^{i, j}_B$ is the bundle of type $(i, j)$ basic forms. We thus have the well defined operators
\[
\begin{split}
\p_B: \Omega^{i, j}_B\rightarrow \Omega^{i+1, j}_B,\\
\bar\p_B: \Omega^{i, j}_B\rightarrow \Omega^{i, j+1}_B.
\end{split}
\]
Then we have $d_B=\p_B+\bar \p_B$. 
Set $d^c_B=\frac{1}{2}\i\left(\bar \p_B-\p_B\right).$ It is clear that
\[
d_Bd_B^c=\i\p_B\bar\p_B=\frac{1}{2}d\Phi d, d_B^2=(d_B^c)^2=0.
\]
We shall  recall the transverse structure on local coordinates.  Let $U_\alpha$ be an open covering of $M$ and $\pi_\alpha: U_\alpha\rightarrow V_\alpha\subset \C^n$ submersions
such that 
\[
\pi_\alpha\circ \pi^{-1}_\beta: \pi_\beta(U_\alpha\cap U_\beta)\rightarrow \pi_\alpha (U_\alpha\cap U_\beta)
\]
is biholomorphic when $U_\alpha\cap U_\beta$ is not empty. One can choose  local coordinate charts $(z_1, \cdots, z_n)$ on $V_\alpha$ and local coordinate charts $(x, z_1, \cdots, z_n)$ on  $U_\alpha\subset M$  such that $\xi=\p_x$, where we use the notations
\[
\p_x=\frac{\p}{\p x}, \p_i=\frac{\p}{\p z_i}, \bar \p_{ j}=\p_{\bar j}=\frac{\p}{\p \bar z_{ j}}=\frac{\p}{\p z_{\bar j}}. 
\]
The map $\pi_\alpha: (x, z_1, \cdots, z_n)\rightarrow (z_1, \cdots, z_n)$ is then the natural projection. There is an isomorphism, for any $p\in U_\alpha$,
\[
d\pi_\alpha: \cD_p\rightarrow T_{\pi_\alpha(p)}V_\alpha. 
\]
The restriction of $g$ on $\cD$ gives an Hermitian metric $g^T_\alpha$ on $V_\alpha$ since $\xi$ generates isometries of $g$. 
One can  verify that there is a well defined K\"ahler metric $g_\alpha^T$ on each $V_\alpha$ and 
\[
\pi_\alpha\circ \pi^{-1}_\beta: \pi_\beta(U_\alpha\cap U_\beta)\rightarrow \pi_\alpha (U_\alpha\cap U_\beta)
\]
gives an isometry of K\"ahler manifolds $(V_\alpha, g^T_\alpha)$.  The collection of K\"ahler metrics $\{g^T_\alpha\}$ on $\{V_\alpha\}$ can be used as an alternative definition of the transverse K\"ahler metric.
  
Now we consider a Sasaki structure induced by a transverse K\"ahler potential $\phi\in \cH$, such that \[\eta_\phi=\eta+\sqrt{-1}(\bar\p-\p) \phi, \omega_\phi=\omega^T+\sqrt{-1}\p\bar\p\phi,\]
where $\eta_\phi, \omega_\phi$ are real basic $(1, 1)$-forms. 
Its transverse scalar curvature $R^T_\phi$ reads, given a local coordinate $(x, z_1, \cdots, z_n)$ with $\xi=\p_x$, 
\begin{equation}
R^T_\phi=-g^{i\bar j}_{T}(\phi)\p_i\p_{\bar j} \log \det (g_{k\bar l}^T+\phi_{k\bar l})
\end{equation}

Notations: for simplicity, we shall use $R, Ric$ to represent the transverse scalar curvature $R^T, Ric^T$, and write $R_\phi$ as the transverse scalar curvature of $\omega_\phi$. 
We $\Delta:=\Delta_g=\Delta_\eta$ to denote the Laplacian operator of $g$, $\Delta_{\eta_\phi}$ to denote the Laplacian operator of the metric \[g_{\eta_\phi}=\eta_\phi\otimes \eta_\phi+g^T_\phi.\]
When acting on basic functions, $\Delta$ and $\Delta_{\eta_\phi}$ coincide with the transverse Laplacian (or basic Laplacian), which we shall denote as $\Delta, \Delta_\phi$ respectively. We shall only keep the notations $\omega^T$ (sometimes $\omega_T$ such as $(\omega_T)^n$) and $g_{i\bar j}^T$ (sometimes $g^{i\bar j}_T$ for its inverse) to indicate that we are working with transverse K\"ahler structures on a Sasaki manifold, instead of working on a K\"ahler manifold. When a function $f$ is not basic, on a local coordinate $(x, z_1, \cdots, z_n)$, we have the following,
\begin{equation}
\Delta_\eta f=\Delta f=\p^2_x f+g^{i\bar j}_T\p_i\p_{\bar j} f. 
\end{equation}

\subsection{The energy functionals in Sasaki setting}

We recall some well-known functionals in Sasaki setting, which are direct extensions of functionals in K\"ahler geometry, see \cite{zhang1} for example.  We use the notations in terms of transverse K\"ahler structure mostly, parallel to K\"ahler setting. 
Let $(M, \xi, \eta, g)$ be a compact Sasaki manifold. We fix the Reeb vector field $\xi$ and the complex structure on its metric cone. Consider the basic K\"ahler class $[\omega_T]=\frac{d\eta}{2}$ and the Sasaki structures induced by 
\begin{equation}
\cH=\{\phi\in C_B^{\infty}(M): \omega_\phi=\omega_T+\sqrt{-1}\p\bar \p \phi>0\}
\end{equation}
We recall the $\mathbb{I}$-functional,
\begin{equation}
\mathbb{I}_{\omega_T}(\phi)=\frac{1}{(n+1)!}\int_M \phi \sum_{k=0}^n \omega^k_T\wedge \omega^{n-k}_\phi\wedge \eta.
\end{equation}
For a given real basic $(1, 1)$ form $\chi$, we also define 
\begin{equation}
\mathbb{I_{{\omega_T}, \chi}}(\phi)=\frac{1}{n!}\int_M \phi \sum_{k=0}^{n-1}\chi\wedge \omega^k_T\wedge \omega_\phi^{n-1-k}\wedge \eta.
\end{equation}
We recall the so-called $\mathbb{J}$-functional, with the base metric $\omega$,
\begin{equation}\label{j0}
\mathbb{J}_{\omega_T}(\phi):=\mathbb{J}(\omega_T, \omega_\phi)=\sum_{k=0}^n\frac{1}{(n+1)!}\int_M \phi \omega^{k}_T\wedge\omega_\phi^{n-k}\wedge \eta-\frac{1}{n!}\int_M \phi\omega^n_\phi\wedge \eta.
\end{equation}
When the base metric $\omega_T$ is clear, we simply write $\mathbb{J}_{\omega_T}=\mathbb{J}$ and $\mathbb{I}_{\omega_T}=\mathbb{I}$. The $\mathbb{J}$-functional  can be characterized by its derivative, 
\begin{equation}
\frac{d \mathbb{J}(\phi)}{d t}=\int_M \frac{\p \phi}{\p t} (\text{tr}_\phi \omega-n) \frac{\omega_\phi^n}{n!}\wedge \eta
\end{equation}
We have the relation
\[
\mathbb{J}(\phi)=\mathbb{I}(\phi)-\frac{1}{n!}\int_M \phi\omega^n_\phi\wedge \eta.
\]
Note that $\mathbb{I}(\phi)$ is a functional on $\phi\in \cH$, while $\mathbb{J}(\phi)$ does not depend on the normalization condition on $\phi$, hence a functional on transverse K\"ahler metrics (or Sasaki metrics).
For a given real $(1, 1)$ form $\chi$, we also define $\mathbb{J}_{\omega, \chi}$ by
\begin{equation}\label{j1}
\begin{split}
\mathbb{J}_{{\omega_T}, \chi}(\phi)=&\frac{1}{n!}\int_M \phi \sum_{k=0}^{n-1}\chi\wedge \omega^k_T\wedge \omega_\phi^{n-1-k}\wedge \eta-\frac{1}{(n+1)!}\int_M \underline{\chi} \phi \sum_{k=0}^n \omega^k_T\wedge \omega_\phi^{n-k}\wedge \eta
\end{split}
\end{equation}
where $\underline \chi$ takes the form
\[
\underline \chi=\text{Vol}^{-1}(M) \int_M \chi\wedge \frac{\omega^{n-1}}{(n-1)!}\wedge \eta. 
\]
We also recall Aubin's $I$-functional and $J$-functional  in Sasaki setting as follows,
\begin{equation}\label{ij1}
\begin{split}
I_{\omega_T}(\phi):=I(\omega_T, \omega_\phi)=\frac{1}{n!}\int_M \phi\left(\omega^n_T-\omega_\phi^n\right)\wedge \eta\\
J_{\omega_T}(\phi):=J(\omega_T, \omega_\phi)=\frac{1}{n!}\int_M\phi \omega^n_T\wedge \eta-\mathbb{I}_{\omega_T}(\phi). 
\end{split}
\end{equation}
We have the following well-known facts, see for example \cite{PG}[Proposition 4.2.1],
\begin{equation}\label{ij2}
0\leq \frac{1}{n+1}I(\phi)\leq J(\phi)\leq \frac{n}{n+1}I(\phi)
\end{equation}
It follows that,
\begin{equation}\label{ij3}
\frac{1}{n+1} I(\phi)\leq \mathbb{J}(\phi)=I(\phi)-J(\phi)\leq \frac{n}{n+1}I(\phi)
\end{equation}
In general, for $\phi, \psi\in \cH$, we write
\begin{equation}\label{I-topology}
I(\phi, \psi)=I(\omega_\phi, \omega_\psi)=\frac{1}{n!}\int_M (\phi-\psi)\left(\omega^n_\psi-\omega_\phi^n\right)\wedge \eta
\end{equation}

Now we recall that the Mabuchi's $\cK$-energy can be characterized by its variation,
\[
\delta K=-\int_M \delta \phi (R_\phi-\underline R)\frac{\omega_\phi^n}{n!}\wedge \eta,
\]
where $\underline{R}$ is the average of the transverse  scalar curvature.
Following Chen \cite{chen01}, the $\cK$-energy reads (in Sasaki setting),
\begin{equation}\label{kenergy0}
\cK(\phi)=H(\phi)+\mathbb{J}_{\omega, -Ric} (\phi),
\end{equation}
where $H(\phi)$ is the entropy \eqref{entropy101} and the $\mathbb{J}_{-Ric}$-functional takes the form
\[
\begin{split}
\mathbb{J}_{-Ric}(\phi)=&\frac{n\underline{R}}{(n+1)!}\int_M \phi \sum_{k=0}^n \omega^k_T\wedge \omega_\phi^{n-k}\wedge \eta-\frac{1}{n!}\int_M \phi \sum_{k=0}^{n-1}Ric\wedge \omega^k_T\wedge \omega_\phi^{n-1-k}\wedge \eta
\end{split}
\]
An extremely important fact about the $\cK$-energy is that it is convex along $C^{1, \bar 1}$ geodesics connecting points in $\cH$ (conjectured by Chen), proved by Berman-Berndtson \cite{BB} (see also  Chen-Li-Paun \cite{CLP}) 
\begin{thm}[Berman-Berndtsson]On a compact K\"ahler manifold, the $\cK$-energy $\cK(\phi_t)$  is convex in $t$ along the $C^{1, 1}$ geodesics $\phi_t$ connecting $\phi_0, \phi_1\in \cH$, for $t\in [0, 1]$. 
\end{thm}

In Sasaki setting, the convexity of $\cK$-energy along $C^{1, \bar 1}$ geodesic has been proved by Ji-Zhang \cite{JZ} and van Coevering \cite{VC}, building up  on \cite{BB} and \cite{GZ1}. 

\subsection{The metric completion $(\overline\cH, d_1)$}

In this section, we summarize results in pluripotential theory we shall prove in \cite{HL}. These results  will be used to connect the properness of $\cK$-energy to existence of cscs, and play an important role in the proof of Theorem \ref{cscs}. 

Following T. Darvas \cite{D1}, we  introduce the notion of $d_p$ distance on $\cH$,
\[
\|\xi\|_{p, \phi}^p=\int_M |\xi|^p\omega_\phi^n\wedge \eta, \forall \xi\in T_\phi\cH=C^\infty_B(M). 
\]
When $p=2$, this is the counterpart of Mabuchi's metric in the space of K\"ahler potentials and it was studied by Guan-Zhang \cite{GZ1} extensively. We shall be mostly interested in $p=2$ and $p=1$. 
One can then define the length of curves in $\cH$ and the distance function $d_p(\phi_0, \phi_1)$ for $\phi_0, \phi_1$ in $\cH$. The later is simply the infimum of the length of the curves connecting $\phi_0$ and $\phi_1$.
We need to extend many profound results in pluripotential theory in K\"ahler setting to Sasaki setting. We need mostly notions in \cite{GZ}, \cite{D2} and in particular \cite{D4}. Part of them has been done by van Covering \cite{VC}, with  main focus on $C^0$ potentials (part of it on $L^\infty$).  We shall extend most of results regarding $\cE$ and $\cE_1$ (defined below) to Sasaki setting.  Given a Sasaki structure $(M, \xi, \eta, g)$, we recall the following definition (see \cite{VC} for example), 
\begin{defn}An $L^1$, upper semicontinuous (usc) function $u: M\rightarrow \R\cup \{-\infty\}$ is called a transverse $\omega^T$-plurisubharmonic (tpsh for short) if $u$ is invariant under the Reeb flow, and $u$ is $\omega^T$-plurisubharmonic on each local foliation chart $V_\alpha$, that is $\omega^T_\alpha+\sqrt{-1}\p\bar\p u\geq 0$ as a positive closed $(1, 1)$-current on $V_\alpha$. 
 \end{defn}

In other words, $\phi$ defines a genuine $\omega^T_\alpha-$plurisubharmonic function with respect to the transverse K\"ahler structure $\omega^T$ on each chart $V_\alpha$. Because of the cocycle conditions, this definition is independent of the choice of foliation charts. Since the definition depends only on transverse K\"ahler structure (not on a particular choice of foliation charts), we can refer 
$\omega_u:=\omega^T+\sqrt{-1}\p_B\bar\p_Bu\geq 0$ as a \emph{positive basic $(1, 1)$ current} (even though we shall not pursue general theory of positive basic currents on Sasaki manifolds).
We use the notation usc to stand for upper semicontinuous, and 
\begin{equation}\label{psh}
\text{PSH}(M, \xi, \omega^T):=\{u\in L^1(M, dv_g), u\; \text{is usc and is invariant w.r.t}\; \xi, \omega_u\geq 0\}
\end{equation}

Since $u$ is a psh function with respect to the K\"ahler form $\omega^T_\alpha$ on $V_\alpha$, notions in local nature for psh functions certainly apply directly to $u$, such as Lelong number etc. We shall use these notions freely.  
Guedj and Zeriahi \cite{GZ} introduced the finite energy space of $\omega$-plurisubharmonic functions, for any $p\geq 1$, on K\"ahler manifolds. Their notions have a direct adaption in Sasaki setting, 
\begin{equation}
\cE_p(M, \xi, \omega^T)=\{\phi\in\text{PSH}(M, \xi, \omega^T): \int_M |\phi|^p\omega_\phi^n\wedge \eta<\infty, \int_M \omega_\phi^n\wedge \eta=\int_M \omega^n_T\wedge \eta\}
\end{equation}

Darvas has proved that $(\cE_p, d_p)$ is a geodesic metric space which is the metric completion of $(\cH, d_p)$ in K\"ahler setting \cite{D1, D2}, together with an effective estimate of the distance $d_p$. 
We shall summarize the results  in \cite{HL}, by extending the pluripotential theory on K\"ahler manifolds to  Sasaki setting.  Darvas's lecture notes \cite{D4} would be a very good reference. 
First we shall follow T. Darvas (see \cite{D4}[Theorem 3.36]) to prove the following in Sasaki setting, 

\begin{thm}[He-Li \cite{HL}]\label{hl01}Given $u_0, u_1\in \cE_1(M, \xi, \omega^T)$, there exists decreasing sequences $u_0^k, u_1^k\in \cH$ such that $u_0^k\rightarrow u_0, u^k_1\rightarrow u_1$. Define the distance $d_1(u_0, u_1)$ as follows,
\begin{equation}d_1(u_0, u_1)=\lim_{k\rightarrow \infty} d_1(u_0^k, u_1^k)
\end{equation}
The above limit exists and it is independent of approximating sequences. For each $t\in (0, 1)$, define
\[
u_t:=\lim u^k_t, t\in (0, 1),
\]
where $u^k_t$ is the $C^{1, \bar 1}$ geodesic connecting $u^k_0$ and $u^k_1$, by the result of  \cite{GZ1}. Then $u_t\in \cE_1$ and the curve $t\rightarrow u_t$ is independent of the choice of approximating sequences and is a $d_1$ geodesic in the sense that \[d_1(u_t, u_s)=d_1(u_0, t_1)|t-s|, t, s\in [0, 1].\] $(\cE_1(M, \xi, \omega^T), d_1)$ is a geodesic metric space which is the metric completion of  $(\cH, d_1)$. Moreover, \eqref{d01} and \eqref{d02} hold. 
\end{thm}

Indeed as in T. Darvas \cite{D2}, the above theorem holds in a much more general setting (for $p\in [1, \infty)$). 
Darvas's results also have strong connections with previous results in pluripotential theory, which we now recall. Note that
Aubin's $I$-function can be used to define a strong topology in $\cE_1$ \cite{BBEGZ}[Section 2]. We shall need the following adaption to Sasaki setting, see \cite{BBEGZ}[Theorem 1.8] and \cite{CC2}[Lemma 5.8],
\begin{lemma}[He-Li \cite{HL}]\label{ifunctional}For $\phi_1, \phi_2, \phi_3\in \cE_1(M, \xi, \omega^T)$, there exists a constant $C=C(n)$ such that
\begin{equation}
I(\phi_1, \phi_3)\leq C(I(\phi_1, \phi_2)+I(\phi_2, \phi_3)). 
\end{equation}
\end{lemma}

We recall the compactness theorem \cite{BBEGZ}[Theorem 2.17]. We have the following adaption to Sasaki setting,

\begin{lemma}[He-Li \cite{HL}]\label{compactness}Suppose $\{u_k\}\subset \cH$ such that $d_1(0, u_k), H(u_k)\leq D$ for some $D\geq 0$. Then there exists $u\in \cE_1$ such that $d_1(u_{k_l}, u)\rightarrow 0$ as $k_l\rightarrow \infty$. 
\end{lemma}

The above lemma simply means that $d_1$ bounded set with bounded entropy is precompact in $(\cE_1, d_1)$.  Given the effective distance estimates, one can extend functionals such as entropy $H$ (see Lemma \ref{approx13}), $\mathbb{J}_{-Ric}$ to $\cE_1$ class as in \cite{BDL}. In particular Berman-Darvas-Lu extended the $\cK$-energy to $(\cE^1, d_1)$ and proved that the $\cK$-energy is convex along the geodesics in $\cE^1$. 
\begin{thm}[Berman-Davars-Lu]\label{bdl}On a compact K\"ahler manifold, the $\cK$-energy can be extended to a functional $\cK: \cE_1\rightarrow \R\cup\{+\infty\}$ using Chen's formula \cite{chen00}
Such a $\cK$-energy in $\cE^1$ is the greatest $d_1$-lsc extension of $\cK$-energy on $\cH$. Moreover, $\cK$-energy is convex along the finite energy geodesics of $\cE^1$. 
\end{thm}

We shall extend the $\cK$-energy to $\cE_1$ and prove its convexity in $\cE_1$, just as in \cite{BDL}. 

\begin{lemma}[He-Li \cite{HL}]\label{hl02}The functional  $\mathbb{J}_{-Ric}$ is continuous with respect to $d_1$ and can be extended continuously to $\cE_1$. The $\cK$-energy can be extended to a functional $\cK: \cE_1\rightarrow \R\cup\{+\infty\}$ using the formula \eqref{kenergy0}.
Such a $\cK$-energy in $\cE^1$ is the greatest $d_1$-lsc extension of $\cK$-energy on $\cH$. Moreover, $\cK$-energy is convex along the finite energy geodesics of $\cE^1$.  
\end{lemma}

\subsection{Reduced properness}

We need a notion of properness of the $\cK$-energy in terms of $d_1$  relative to group action as follows. Let $G:=\text{Aut}_0(\xi, J)$ be the identity component of the automorphism group which preserves both Reeb vector field and transverse complex structure.  Let $G_0$ be the identity component of the reduced part of $\text{Aut}_0(\xi, J)$; that is, its Lie algebra consisting exactly complex Hamiltonian holomorphic vector fields on $M$. 
Denote $\cH_0$ to be the space of normalized transverse K\"ahler potentials, 
\begin{equation}\label{sntkp}
\cH_0=\{\phi\in \cH: \mathbb{I}(\phi)=0\}.
\end{equation}
We describe now how $G$ acts on $\cH_0$. Given $\sigma\in G$, since $\sigma$ preserves $\xi$ and transverse holomorphic structure, $(\xi, \sigma^*\eta, \sigma^* g)$ is a Sasaki structure with the same transverse holomorphic structure. By \cite{BG}, there exists basic functions $\phi, \psi$ and an harmonic 1-form $\alpha$ (with respect to $g$) such that
\[
\begin{split}
&\sigma^*\eta=\eta+d^c_B \phi+d\psi_1+\alpha\\
&\sigma^*\Phi=\Phi-(\xi\otimes (\sigma^*\eta-\eta)\circ \Phi)\\
&\sigma^* g= \sigma^*\eta\otimes \sigma^*\eta+\frac{1}{2}d(\sigma^*\eta)(\mathbb{I}\otimes \sigma^{*}\Phi).
\end{split}
\]
The transverse K\"ahler form $\sigma^*(\omega^T)=\frac{1}{2}d(\sigma^*\eta)=\frac{1}{2}d\eta+dd^c_B\phi=\omega^T+\sqrt{-1}\p_B\bar\p_B \phi$ remains in the same basic cohomology class $[\omega^T]$. The change induced by $\psi_1, \alpha$ does not affect the underlying metric; hence we consider $\sigma^*(\omega^T)$ given $\sigma\in G$ only. For any transverse K\"ahler form $\omega_\phi$, then $\sigma^*{\omega_\phi}$ gives another transverse K\"ahler form in $[\omega^T]$, we define its transverse K\"ahler potential as $\sigma [\phi]\in \cH_0$, such that
\begin{equation}
\sigma^*{\omega_\phi}=\omega^T+\sqrt{-1}\p_B \bar\p_B \sigma[\phi], \sigma[\phi]\in \cH_0. 
\end{equation} 
We write $\sigma^*(\omega^T)=\omega^T+\sqrt{-1}\p_B\bar\p_B \sigma[0]$, with the normalization condition  $\sigma[0]\in \cH_0$. As in K\"ahler setting (see \cite{DR}[Lemma 5.8]), we have the following,
\begin{prop}Given $\phi\in \cH_0$, the $G$-action is given by, 
 $\sigma[\phi]=\sigma[0]+\phi\circ \sigma$. 
 \end{prop}
 \begin{proof}We have the following, that 
 \[\sigma^*(\omega_\phi)=\sigma^{*}(\omega^T)+\sqrt{-1}\p_B\bar\p_B \sigma^*\phi.\]
 Hence  $\sigma[\phi]$ differs $\sigma[0]+\phi\circ \sigma$ by a constant. 
 Note that given $u, v\in \cH$,
 \[
 \mathbb{I}(u)-\mathbb{I}(v)=\frac{1}{(n+1)!}\int_M (u-v)\omega_u^k\wedge \omega_v^{n-k}\wedge \eta
 \]
 We compute
\[
\begin{split}
\mathbb{I}(\sigma[0]+\phi\circ \sigma)-\mathbb{I}(\sigma[0])=&\frac{1}{(n+1)!}\sum_k\int_M \phi\circ \sigma (\sigma^*\omega_T)^k\wedge(\sigma^* \omega_\phi)^{n-k}\wedge \eta\\
=&\frac{1}{(n+1)!}\sum_k\int_M \phi \circ\sigma (\sigma^*\omega_T)^k\wedge(\sigma^* \omega_\phi)^{n-k}\wedge (\sigma^*\eta)\\
=&\mathbb{I}(\phi)=0.
\end{split}\]
It follows that $\sigma[0]+\phi\circ \sigma\in \cH_0$ and $\sigma[\phi]=\sigma[0]+\phi\circ \sigma$.
 \end{proof}
 
It is well-known that $G$ acts on $\cH_0$ by isometry with respect to the Mabuchi metric (in K\"ahler setting); and it holds also for $d_1$, see \cite{DR}[Lemma 5.9].
\begin{prop}The action $G$ on $\cH_0$ is an isometry with respect to $d_1$ (and $d_2$) and it has a unique $d_1$-isometric extension to the metric completion $(\cE_1\cap \mathbb{I}^{-1}(0), d_1)$.
\end{prop}
For any given transverse K\"ahler metric $\omega_0\in [\omega^T]$, we consider its $G$-orbit
\[
\cO_{\omega_0}=\{\phi\in \cH_0| \sigma^*\omega_0=\omega_\phi, \;\sigma\in G\}
\] 
Given any other transverse K\"ahler metric $\omega_\varphi$ such that $\varphi\in \cH_0$,
the distance function \[d_1(\varphi, \cO_{\omega_0})=\inf_{\omega_\psi\in \cO_{\omega_0}} d_1(\varphi, \psi)\]
Note that since $G$ acts on $\cH_0$ isometrically, we know that
\[
\inf_{\omega_\psi\in \cO_{\omega_0}} d_1(\varphi, \psi)=\inf_{\sigma\in G} d_1(\varphi, \sigma[\psi])=\inf_{\sigma_1, \sigma_2\in G} d_1(\sigma_1 [\varphi], \sigma_2[\psi]). 
\]
Hence we define the $d_1$ distance relative to the action of $G$ as \[d_{1, G}(\varphi, \phi)=\inf_{\sigma\in G} d_1(\varphi, \sigma[\psi]),\]
where $\sigma[\psi]\in \cH_0$ denotes the transverse K\"ahler potential of $\sigma^*\omega_\psi$.
\begin{defn}[Reduced properness]\label{rp}The $\cK$-energy is proper with respect to $d_{1, G}$ if the following conditions hold, 
\begin{enumerate}
\item The Futaki invariant is zero and the $\cK$-energy is bounded from below on $\cH_K$, where $\cH_K$ is $K$-invariant transverse K\"ahler potentials and $K$ is a maximal compact subgroup of $G$ containing the flow of $\xi$.
\item For any sequence $\phi_i\in \cH^0_K$, $d_{1, G}(0, \phi_i)\rightarrow \infty$ implies $\cK(\phi_i)\rightarrow \infty$. 
\end{enumerate}
\end{defn}

\begin{rmk}We recall the definition of properness in K\"ahler setting. The notion of properness of $\cK$-energy was first introduced by G. Tian \cite{Tian94}, with the properness of $\cK$-energy in terms of Aubin's $J$-functional. Tian \cite{Tian97} proved that the properness of $\cK$-energy is equivalent to existence of K\"ahler-Einstein metric on Fano manifolds (see \cite{DR} for clarification of properness when the automorphism group is not discrete).  Later on Tian \cite{Tian01} has made a properness conjecture on existence of csck. Chen has made a properness conjecture with the properness of $\cK$-energy in terms of Mabuchi's metric $d_2$, in line with Donaldson's program \cite{Don97}. A general properness conjecture was then clarified by Darvas-Rubinstein \cite{DR} in terms of $d_1$ and also Aubin's $J$-functional, modulo the action of automorphism of the K\"ahler structure. 
Recently Chen-Cheng \cite{CC1, CC2, CC3} proved very deep results that properness implies the existence of csck on a compact K\"ahler manifolds. 
Our definition of reduced properness is a direct adaption of K\"ahler setting \cite{He18}. \end{rmk}

We need the following is well-known facts \cite{Fut, M2, FM} (in K\"ahler geometry). Its extension to Sasaki setting is tautology. We shall skip the details. 
\begin{lemma}\label{fm02}The $\cK$-energy is $G$ invariant if and only if the Futaki invariant is zero.
\end{lemma}

\begin{rmk}If we assume that $\cK$-energy is bounded below over $\cH$, then the Futaki invariant being zero follows as a consequence. Since we want to test properness for only invariant metrics, it is necessary for us to assume that the Futaki invariant is zero as a part of definition. When we consider properness, it is essentially equivalent to use the group $G$ or its subgroup $G_0$ . 
\end{rmk}

\section{Scalar curvature type equation on Sasaki manifolds}

We consider the scalar curvature type equation, for a real basic $(1, 1)$-form $\beta$, and a bounded basic function $h$, such that
\begin{equation}
R_\phi=\text{tr}_\phi \beta+h,
\end{equation}
where $R_\phi$ is the scalar curvature of the metric $\omega_\phi$. 
Write this equation as follows,
\begin{equation}\label{scalar1}
\begin{split}
&\det(g_{i\bar j}^T+\phi_{i\bar j})=e^F\det(g_{i\bar j}^T)\\
&\Delta_\phi F=\text{tr}_\phi (Ric-\beta)-h
\end{split}
\end{equation} 
Sometimes it is more convenient to write the above as,  with $\gamma=Ric-\beta$, 

\begin{equation}\label{scalar2}
\begin{split}
&\det(g_{i\bar j}^T+\phi_{i\bar j})=e^F\det(g_{i\bar j}^T)\\
&\Delta_\phi F=\text{tr}_\phi \gamma-h
\end{split}
\end{equation} 
In this section we derive a priori estimates for \eqref{scalar2}, mimicking Chen-Cheng \cite{CC1, CC2, CC3}. Technically we mainly follow \cite{CC3}[Section 2]. There are some necessary modifications, in particular for $C^0$ estimates of $\phi$ and $F$. On the other hand, $L^p$ estimate of $n+\Delta \phi$ and estimate of $|\nabla_\phi F|_\phi$, the proof is purely by integral method and it can be carried over almost identically to the Sasaki setting, since the quantities we will be computing are all \emph{basic}, and the computations are identically to K\"ahler setting in dimension $2n$ via transverse K\"ahler structure. There are exceptions of course. For example when we use Sobolev inequalities, the dimension of the manifold is $2n+1$, instead of $2n$. Nevertheless we will present these estimates of \eqref{scalar2} in details. 
 
We denote $C_0$ to be a positive constant which depends  the upper bound of $H(\phi)$, the bounds of $\max |\gamma|_{\omega^T}, \max |h|$ and the background metric $(M, g)$. 
From line to line constants differ even with the same labelling such as $C_1$ can vary line by line  but they are all  uniformly bounded, with the dependence specified above. We also write $C=C(C_1, C_2, C_3, \cdots)$ to denote a uniform constant $C$ which depends on the arguments $C_1, C_2, C_3, \cdots$
We have the following, 
\begin{thm}\label{main0}
Let $(M, g, \xi, \eta)$ be a compact Sasaki manifold. Consider a basic function $\phi\in \cH, \sup \phi=0$ satisfying \eqref{scalar2}. Then there exists a constant $C_0>1$ such that
\begin{equation}
C_0^{-1}\leq n+\Delta \phi\leq C_0, \|\phi\|_{C^k}\leq C=C(k, C_0). 
\end{equation}
\end{thm}

This theorem corresponds to \cite{CC2}[Theorem 3.1]; our argument follows closely \cite{CC3}[Section 2, Theorem 2.3] and the method indeed simplifies arguments in \cite{CC2}. 

\subsection{$C^0$-estimate of $\phi$ and $F$}

We assume the normalization condition $\sup \phi=0$ unless specified otherwise. 
The main theorem in this section is as follows, 
\begin{thm}\label{firstorder}There exists a uniformly bounded constant $C_0$ such that
\begin{equation}
|\phi|+|F|\leq C_0,
\end{equation}
where $C_0$ depends on the upper bound of entropy $\int_M F e^F dv_g$, the bound $\max |\gamma|_{\omega^T}, \max |h|$ and the background metric $(M, g)$. 
\end{thm}

We state a version of Alexanderov maximum principle, which will be used in a substantial way to obtain $C^0$ estimates of $\phi$ and $F$. 
\begin{lemma}\label{alexanderov}
Let $u\in C^2(\Omega)\cap C(\overline \Omega)$ for a bounded domain $\Omega\subset \R^k$. Then there exists a constant $C_k$ depending only on $k$ such that,
\[
\sup_\Omega u-\sup_{\p \Omega}\leq C_k \left(\int_{\Gamma^{-}} \det (-D^2u) dx\right)^{\frac{1}{k}},
\]
where $\Gamma^{-}\subset \Omega$ is defined to be,
\[
\Gamma^{-}=\left\{x\in \Omega: u(y)\leq u(x)+\nabla u(x)\cdot (y-x), \text{for any}\; y \in \Omega, |\nabla u(x)|\leq \frac{\sup_\Omega u-\sup_{\p \Omega} u}{3\text{diam}(\Omega)}\right\}
\]
\end{lemma}

We have the following $C^0$ estimate of $\phi$, 
\begin{lemma}\label{cma0}Let $\phi$ and $F$ be basic functions such that
\begin{equation}
\det(g_{i\bar j}^T+\phi_{i\bar j})=e^F\det(g_{i\bar j}^T), 
\end{equation}
Then there exists a uniform constant $C_1$ such that
\[
|\phi|\leq C_1(\|e^F\|_{L^{2(1+\epsilon)}}, \epsilon, M, g)
\]
\end{lemma}
\begin{proof}
This is a modification of $C^0$ estimate for complex Monge-Ampere equation using Alexandrov maximum principle, which is due to Blocki \cite{B05}. First note that $\Delta_\eta \phi> -n$ and $\sup \phi=0$ implies that
\[
\int_M |\phi| dv_g\leq C(M, g).
\]
Suppose $-\inf_M\phi =L>0$ is achieved at a point $p\in M$. We can choose a coordinate $(x, z_1, \cdots, z_n)$ around $p$ such that  the metric $g_{\eta}=\eta\otimes \eta+g_{i\bar j}^Tdz_i\otimes d\bar z_j$ has the form
\[
g(p)=(dx)^2+\sum dz_i\otimes d\bar z_i
\]
and $2^{-1}\delta_{i\bar j}\leq g_{i\bar j}^T\leq 2\delta_{i\bar j}$ holds in the ball $B_r(p)$ for a fixed $r>0$ sufficiently small.  Consider $v=\phi+2^{-1}(|z|^2+x^2)$, a local function defined on $B_r(p)$. Denote $\Gamma^{-}=\Gamma^{-}(-v, B_r(p))$. Then $D^2 v\geq 0$ on $\Gamma^{-}$. It follows that $\det(D^2 v)\leq (\det(v_{i\bar j}))^2$. Applying Lemma \ref{alexanderov} to $-v$, we get
\[
2^{-1} r^2\leq \sup_B(-v)-\sup_{\p B}(-v)\leq C_n\left(\int_{\Gamma^{-}} \det(D^2v) dv_{\text{Euc}}\right)^{\frac{1}{2n+1}},
\]
where $dv_{\text{Euc}}$ is the Euclidean volume form. Note that $2^{-1}\delta_{i\bar j}\leq g_{i\bar j}^T\leq 2\delta_{i\bar j}$, it follows that
\[
2^{-1}r^2\leq C_1(n)\left(\int_{\Gamma^{-}}  (\det(v_{i\bar j}))^2dv_{g}\right)^{\frac{1}{2n+1}},
\]
Note that over $B_r(p)$, we have
\[
v_{i\bar j}=\phi_{i\bar j}+2^{-1}\delta_{i\bar j}\leq g_{i\bar j}^T+\phi_{i\bar j}
\]
It then follows that, by H\"older inequality, 
\begin{equation}\label{vol1}
\begin{split}
2^{-1}r^2\leq& C_2(n)\left(\int_{\Gamma^{-}}e^{2F} dv_{g}\right)^{\frac{1}{2n+1}}\\
\leq &C_2(n)\left(\int_{\Gamma^{-}} e^{2(1+\epsilon)F}dv_{g}\right)^{\frac{1}{(1+\epsilon)(2n+1)}}\left(\int_{\Gamma^{-}} dv_{g}\right)^{\frac{\epsilon}{(1+\epsilon)(2n+1)}}\\
\leq & C_2(n) \|e^{2F}\|_{L^{2(1+\epsilon)}}^{\frac{1}{2n+1}}\text{Vol}(\Gamma^{-})^{\frac{\epsilon}{(1+\epsilon)(2n+1)}}
\end{split}
\end{equation}
Now we consider for any $q\in \Gamma^{-}$, 
\begin{equation}
(-v)(0)\leq (-v)(q)+\nabla(-v)(q)\cdot (0-q)
\end{equation}
In other words,
\begin{equation}\label{vol2}
\begin{split}
-L=v(0)\geq& v(q)-\nabla v(q) \cdot q\geq v(q)-\frac{\sup_B(-v)-\sup_{\p B}(-v)}{6r} r\\
\geq& v(z)-\frac{1}{6}(L+2^{-1}r^2)
\end{split}
\end{equation}
It follows that, for $q\in \Gamma^{-}$, 
\begin{equation}\label{vol3}
\phi(q)\leq v(q)\leq \frac{-5L}{6}+\frac{r^2}{12}
\end{equation}
It follows that
\begin{equation}\label{vol4}
 \text{Vol}(\Gamma^{-}) \left(\frac{5L}{6}-\frac{r^2}{12}\right)\leq \int_{\Gamma^{-}} -\phi dv_{g}\leq \int_M |\phi|dv_{g}\leq C(M, g)
\end{equation}
By \eqref{vol1} and \eqref{vol4}, we get that 
$L\leq C_3.$
This completes the proof.
\end{proof}

Chen-Cheng's following construction of auxiliary function $\psi$ is the key to obtain $C^0$ estimates of $\phi$ and $F$. We adapt it to the Sasaki case. 

\begin{lemma}\label{cma1}
Let $\psi$ be the solution of the following equation with $\sup\psi=0$,
\begin{equation}\label{cma2}
\det(g_{i\bar j}^T+\psi_{i\bar j})= \underline Ce^F \sqrt{F^2+1}\det(g_{i\bar j}^T),
\end{equation}
where the constant $\underline C$ is given by, 
\[
\underline C= A^{-1}\int_M dv_g, A:=\int_M e^F \sqrt{F^2+1}\;dv_g
\]
Then for any $\epsilon_0\in (0, 1)$, there exist constant $C, C_1$ such that,
\begin{equation}
F+\epsilon_0 \psi-2C\phi\leq C_1,
\end{equation}
where we have $C=2(\max |\gamma|_{\omega^T}+1)$ and $C_1=C_1(\epsilon_0, C_0)$. 
\end{lemma}

\begin{proof}
We compute, with $C=2(\max |\gamma|_{\omega}+1)$
\begin{equation}
\Delta_\phi (F+\epsilon_0\psi-C\phi)=\Lambda_\phi\gamma+\epsilon_0 \Delta_\phi \psi-Cn+C\Lambda_\phi \omega^T.
\end{equation}
By arithmetic-geometric-inequality, we have
\begin{equation}
\Delta_\phi \psi=\Lambda_\phi \omega_\psi-\Lambda_\phi \omega^T\geq \left(A^{-1}\sqrt{F^2+1}\right)^{\frac{1}{n}}-\Lambda_\phi \omega^T. 
\end{equation}
It follows that
\begin{equation}\label{u01}
\Delta_\phi (F+\epsilon_0\psi-C\phi)\geq \frac{C}{2} \Lambda_\phi \omega^T+\epsilon_0A^{-\frac{1}{n}} (F^2+1)^{\frac{1}{2n}}-nC
\end{equation}
Now denote $u=F+\epsilon_0\psi-C\phi$. Choose constants $\delta, \theta\in (0, 1)$ which are specified later. For any $p\in M$, we can construct a smooth function $\zeta_p$ such that $\zeta_p(p)=1$ and $\zeta_p\equiv 1-\theta$  outside a fixed geodesic ball $B_r(p)$ of the background metric $(M, g)$  for $0<r$ sufficiently small with the control on derivatives,  
\begin{equation}\label{cutoff}
|\nabla\zeta_p|\leq 2\theta r^{-1}, |\nabla^2\zeta_p|\leq 2\theta r^{-2}.
\end{equation}
We identify $B_r(p)$ with a Euclidean ball when $r$ is sufficiently small (less than injectivity radius of $(M, g)$). 
Now suppose $u$ achieves its maximum at the point $p_0$. We want to compute $\Delta_{\eta_\phi} (e^{\delta u}\zeta_{p_0})$. We first explain the  differences with the K\"ahler setting. $\Delta_{\eta_\phi}$ is the metric Laplacian of  the metric determined by $(\eta_\phi, d\eta_\phi)$.  We shall also note that $u$ is a basic function but $\zeta_{p_0}$ is not a basic function. 
We compute
\begin{equation}\label{c01}
\Delta_{\eta_\phi}(e^{\delta u}\zeta_{p_0})=\Delta_{\eta_\phi}(e^{\delta u}) \zeta_{p_0}+e^{\delta u} \Delta_{\eta_\phi} \zeta_{p_0}+2 \nabla_{\eta_\phi} e^{\delta u}\cdot  \nabla_{\eta_\phi} \zeta_{p_0}
\end{equation}
Since $u$ is a basic function, we have
\begin{equation}\label{c02}
\Delta_{\eta_\phi}(e^{\delta u})=\Delta_\phi(e^{\delta u})=e^{\delta u} \delta^2 |\nabla_\phi u|^2_\phi+e^{\delta u} \delta \Delta_\phi u. 
\end{equation}
We compute
\begin{equation}\label{c03}
\Delta_{\eta_\phi} \zeta_{p_0}=\nabla^2_{\xi, \xi}(\zeta_{p_0})+\Delta_\phi (\zeta_{p_0})\geq -(1+\Lambda_\phi\omega^T) |\nabla^2\zeta_{p_0}|
\end{equation}
We also compute, noting that $u$ is basic,
\begin{equation}\label{c04}
\begin{split}
2 \nabla_{\eta_\phi} e^{\delta u}\cdot  \nabla_{\eta_\phi} \zeta_{p_0}&=2\delta e^{\delta u}\nabla_\phi u\cdot_\phi \nabla_\phi \zeta_{p_0}\\
&\geq -\delta^2 e^{\delta u}|\nabla_\phi u|^2_\phi \zeta_{p_0}-e^{\delta u}|\nabla_\phi \zeta_{p_0}|^2 \zeta_{p_0}^{-1}\\
&\geq -\delta^2 e^{\delta u}|\nabla_\phi u|^2_\phi \zeta_{p_0}-e^{\delta u}\Lambda_\phi \omega^T |\nabla \zeta_{p_0}|^2 \zeta_{p_0}^{-1}
\end{split}
\end{equation}
Putting \eqref{c01}, \eqref{c02}, \eqref{c03} and \eqref{c04} together, we have 
\begin{equation}\label{c05}
\Delta_{\eta_\phi}(e^{\delta u}\zeta_{p_0})\geq e^{\delta u} \delta (\Delta_\phi u) \zeta_{p_0}-e^{\delta u}(1+\Lambda_\phi\omega^T) |\nabla^2\zeta_{p_0}|-e^{\delta u}\Lambda_\phi \omega^T |\nabla \zeta_{p_0}|^2\zeta_{p_0}^{-1}
\end{equation}
We use \eqref{u01} and \eqref{c05} to get, 
\begin{equation}\label{c06}
\begin{split}
\Delta_{\eta_\phi}(e^{\delta u}\zeta_{p_0})\geq & e^{\delta u}(\delta \zeta_{p_0} \frac{C}{2}- |\nabla^2\zeta_{p_0}|-\frac{|\nabla \zeta_{p_0}|^2}{\zeta_{p_0}}) \Lambda_\phi \omega^T\\&\;\;+e^{\delta u}\delta \zeta_{p_0}(\epsilon_0A^{-\frac{1}{n}} (F^2+1)^{\frac{1}{2n}}-nC)-e^{\delta u}|\nabla^2\zeta_{p_0}|
\end{split}
\end{equation}
We choose $\theta$ sufficiently small such that  (using \eqref{cutoff}),
\begin{equation}\label{theta01}
\delta \zeta_{p_0} \frac{C}{2}- |\nabla^2\zeta_{p_0}|-\frac{|\nabla \zeta_{p_0}|^2}{\zeta_{p_0}}\geq \delta (1-\theta)\frac{C}{2}-2\theta r^{-2}-4\theta^2 r^{-2} (1-\theta)^{-1}>0
\end{equation}
With such a choice of $\theta=\theta(\delta, C, r)$, \eqref{c06} gives
\begin{equation}\label{c07}
\Delta_{\eta_\phi}(e^{\delta u}\zeta_{p_0})\geq e^{\delta u}\delta \zeta_{p_0}(\epsilon_0A^{-\frac{1}{n}} (F^2+1)^{\frac{1}{2n}}-(n+1)C),
\end{equation}
where we have used the fact that $|\nabla^2\zeta_{p_0}|<\delta \zeta_{p_0}C$ by \eqref{theta01}. Now we want to apply Alexanderov maximum principle to \eqref{c07}, using the operator $\Delta_{\eta_\phi}$. Given a local coordinate $(x, z_1, \cdots, z_n)$ of $B_r(p_0)$, we have
\[
\Delta_{\eta_\phi} h=\p^2_xh+g^{i\bar j}_\phi\frac{\p^2}{\p z_i\p \bar z_j} h\geq f,
\]
then the Alexanderov maximum principle asserts that
\begin{equation}\label{alx}
\sup_B h-\sup_{\p B} h \leq C_n r\left\|\frac{f^{-}}{D^*}\right\|_{L^{2n+1}(B)}, \;D^*=\det(g^{i\bar j}_\phi)^{\frac{2}{2n+1}}.
\end{equation}
Denote $v=\epsilon_0 A^{-\frac{1}{n}}(F^2+1)^\frac{1}{2n}-C_1$. There exists a positive constant $C_2$ such that $F\geq C_2=C_2(\epsilon, A, C_1)$, then $v\geq 0$.
Applying \eqref{alx} to \eqref{c07}, we have ($\zeta_{p_0}(p_0)=1$),
\begin{equation}\label{c08}e^{\delta u(p_0)}\leq \sup_{\p B_r(p_0)} e^{\delta u}\zeta_{p_0}+C \delta r\left(\int_{F\leq C_2}e^{2F} e^{(2n+1)\delta u}(v^{-})^{2n+1}dv_g\right)^{\frac{1}{2n+1}}
\end{equation}
Evidently $v^{-}\leq C_1$, hence it follows that
\begin{equation}\label{c09}
\begin{split}
\int_Be^{2F} e^{(2n+1)\delta u}(v^{-})^{2n+1}dv_g\leq& \int_{{F\leq C_2}} e^{2F+2n\delta F} e^{-2nC\delta \phi} (v^{-})^{2n+1}dv_g\\
\leq C_3\int_M e^{-2nC\delta \phi} dv_g
\end{split}
\end{equation}
Now we choose $\delta$ such that $2nC\delta<\alpha$, the $\alpha$-invariant in Sasaki setting, such that \[\int_M e^{-2nC\delta \phi} dv_g\leq C_4.\] Hence it follows that,
\[
e^{\delta u(p_0)}\leq (1-\theta) e^{\delta u(p_0)}+C_5,
\]
where $C_5$ depends on $\epsilon_0, A, C_0$. Since $A$ is bounded above by $H$ plus a uniform constant, the proof is complete. 
\end{proof}

\begin{lemma}\label{upper01}There exists a constant $C_0$ such that
\begin{equation}
F\leq C_0, -\phi\leq C_0.
\end{equation}
\end{lemma}
\begin{proof}By Lemma \ref{cma1}, we have
\[F+\epsilon_0 \psi-C\phi\leq C_1.\]
In other words, we have
\[
\frac{\alpha}{\epsilon_0}(F-C\phi)
\leq -\alpha \psi+C_1\frac{\alpha}{\epsilon_0}\]
It follows that, for any $p\geq 1$, choose $\epsilon_0>0$ such that $p=\alpha \epsilon^{-1}_0$, we get
\[\int_M e^{pF}dv_{g}\leq e^{pC_1}\int_M e^{-\alpha\psi} dv_{g}\leq C_4.\]
Applying Lemma \ref{cma0} (and to \eqref{cma2} also), we get that
$|\phi|, |\psi|\leq C_5$. It follows that $F\leq C_6$. 
\end{proof}

\begin{lemma}\label{lower001}There exists a constant $C_0$ such that
\[
F\geq -C_0
\]
\end{lemma}
\begin{proof}This can be done by a rather direct maximum principle argument,  see \cite{HZ}[Proof of Theorem 1.7]. We consider, at point $p$ where $F+C\phi$ obtains its minimum,
\[\Delta_\phi (F+C\phi)\geq 0\]
It follows that at point $p$,
\[
\Lambda_{\phi}(C\omega^T-\gamma)\leq C_0
\]
Taking $C$ sufficiently large, it follows that at $p$, $F(p)\geq -C_1$. Since $\phi$ is uniformly bounded, we complete the proof. 
\end{proof}

The estimates above finish the proof of Theorem \ref{firstorder}.

\subsection{$L^p$ estimate of $(n+\Delta\phi)$}
In this section we prove $L^p$ estimate of $n+\Delta_B\phi$ for any $p$. Since $\phi$ is basic, $\Delta \phi:=\Delta_{g} \phi=\Delta_B\phi$. 
\begin{thm}\label{pmain}There exists a constant $C_p$ such that for any $p\geq 1$,
\begin{equation}
\int_M (n+\Delta\phi)^p dv_g\leq C_p=C(C_0, p)
\end{equation}
\end{thm}

\begin{proof}
We compute, for constants $\l, C>0$ specified later,
\begin{equation}\label{p01}
\begin{split}
\Delta_\phi (e^{-\l(F+C\phi)}(n+\Delta \phi))=&\Delta_\phi(e^{-\l (F+C\phi)}) (n+\Delta \phi)+e^{-\l (F+C\phi)} \Delta_\phi (\Delta \phi)\\
&-2\l e^{-\l(F+C\phi)} \nabla_\phi (F+C\phi) \cdot_\phi \nabla_\phi (\Delta \phi)
\end{split}
\end{equation}
We compute
\begin{equation}\label{p02}
\begin{split}
\Delta_\phi e^{-\l(F+C\phi)}=&e^{-\l (F+c\phi)}\left(\l^2 |\nabla_\phi (F+C\phi)|^2_\phi-\l \Delta_\phi(F+C\phi)\right)\\
=&e^{-\l (F+c\phi)}\left(\l^2 |\nabla_\phi (F+C\phi)|^2_\phi-\l (Cn-h)+\l\Lambda_\phi(C\omega^T-\gamma) \right)\\
\geq&e^{-\l (F+C\phi)}\left(\l^2 |\nabla_\phi (F+C\phi)|^2_\phi-\l C_1+\frac{\l C}{2}\Lambda_\phi \omega^T\right)
\end{split}
\end{equation}
When we do computations in a coordinate, we can choose a local coordinate around point $p$ such that
\[
g_{i\bar j}^T(p)=\delta_{i\bar j}, \p_k g_{i\bar j}^T(p)=0, \phi_{i\bar j}=\delta_{i\bar j}\phi_{i\bar i}. 
\]
Similar as in K\"ahler case \cite{Yau}[(2.7) (2.10)], we have
\begin{equation}\label{p03}
\begin{split}
\Delta_\phi(\Delta_B\phi)=&R_{i\bar ij\bar j}\frac{1+\phi_{i\bar i}}{1+\phi_{j\bar j}}+\frac{|\phi_{i\bar j k}|^2}{(1+\phi_{i\bar i})(1+\phi_{i\bar j})}+\Delta F-\sum_{i, j} R_{i\bar ij \bar j}\\
\geq &-C_2\Lambda_\phi \omega^T (n+\Delta \phi)+\frac{|\phi_{i\bar j k}|^2}{(1+\phi_{i\bar i})(1+\phi_{j\bar j})}+\Delta F-R^T,
\end{split}
\end{equation}
where $C_2$ depends only on the  curvature of the background metric $(M, g, \omega^T)$ and $R^T$ is he transverse scalar curvature. Observe that
\[
\begin{split}
\frac{|\phi_{i\bar j k}|^2}{(1+\phi_{i\bar i})(1+\phi_{i\bar j})}=&\sum_{i,j,k}\frac{|\phi_{k\bar j i}|^2}{(1+\phi_{i\bar i})(1+\phi_{j\bar j})}
\geq \sum_{i, j=k}\frac{|\phi_{k\bar k i}|^2}{(1+\phi_{i\bar i})(1+\phi_{k\bar k})}\\
\geq &\sum_{i, k}\frac{|\phi_{k\bar k i}|^2}{(1+\phi_{i\bar i}) (n+\Delta\phi)}\geq \frac{ |\nabla_\phi \Delta \phi|_\phi^2}{n+\Delta \phi}.
\end{split}
\]
Hence we compute,
\begin{equation}\label{p04}
\begin{split}
&\l^2 |\nabla_\phi (F+C\phi)|^2_\phi (n+\Delta\phi)+\frac{|\phi_{i\bar j k}|^2}{(1+\phi_{i\bar i})(1+\phi_{j\bar j})}-2\l \nabla_\phi (F+C\phi) \cdot_\phi \nabla_\phi (\Delta \phi)\\
&\;\;\;\;\geq\l^2 |\nabla_\phi (F+C\phi)|^2_\phi (n+\Delta\phi)+\frac{ |\nabla_\phi \Delta \phi|_\phi^2}{n+\Delta \phi}-2\l \nabla_\phi (F+C\phi) \cdot_\phi \nabla_\phi (\Delta \phi)\geq 0
\end{split}
\end{equation}
Combining \eqref{p01}, \eqref{p02}, \eqref{p03} and \eqref{p04}, we have
\begin{equation}\label{p05}
\Delta_\phi (e^{-\l(F+C\phi)}(n+\Delta \phi))\geq e^{-\l (F+C\phi)}\left((\frac{\l C}{2}-C_2)\Lambda_\phi \omega^T (n+\Delta \phi)+\Delta F-C_7\right),
\end{equation}
where the constant $C_7=\l C_1+\max R^T.$ We choose constant $\l\geq 1, C\geq 4C_2$ and denote $u=e^{-\l(F+C\phi)}(n+\Delta \phi)$. We compute, using $dv_\phi:=\eta_\phi \wedge (d\eta_\phi)^n$
\begin{equation}\label{p07}
\begin{split}
\int_M (p-1)u^{p-2}|\nabla_\phi u|_\phi^2 dv_\phi= &\int_Mu^{p-1}(-\Delta_\phi u)dv_\phi\\
\leq & -\int_M u^{p-1} \left(\frac{\l C}{4} u\Lambda_\phi \omega^T+e^{-\l(F+C\phi)}(\Delta F-C_7)\right)dv_\phi
\end{split}
\end{equation}
Now we deal with the term involved with $\Delta F$,
\begin{equation}\label{p08}
\begin{split}
\int_M u^{p-1} e^{-\l (F+C\phi)}\Delta F dv_\phi=&\int_M u^{p-1} e^{(1-\l)F-\l C\phi} \Delta F dv_g\\
=&\int_M u^{p-1}  e^{(1-\l)F-\l C\phi}\Delta (F-\frac{\l C\phi}{1-\l}) dv_g\\
&+\int_Mu^{p-1}  e^{-\l (F+C\phi)}\frac{\l C\Delta \phi}{1-\l} dv_\phi
\end{split}
\end{equation}
We deal with the two terms in \eqref{p08} separately. Denote $v=(1-\l)F-\l C\phi$,
\begin{equation}\label{p09}
\begin{split}
\int_M u^{p-1}  e^{(1-\l)F-\l C\phi}\Delta (F-\frac{\l C\phi}{1-\l}) dv_g=&\frac{1}{1-\l}\int_M u^{p-1} e^v \Delta v dv_g\\
=&\int_M \frac{\nabla (u^{p-1}e^v)\nabla v}{\l-1} dv_g\\
=&\int_M \frac{e^v((p-1)u^{p-2}\nabla u\nabla v+u^{p-1}|\nabla v|^2)}{\l-1} dv\\
\geq &-\int_M \frac{(p-1)^2}{4(\l-1)} e^vu^{p-3}|\nabla u|^2 dv_g\\
\geq&-\int_M  \frac{(p-1)^2}{4(\l-1)} u^{p-2}|\nabla_\phi u|_\phi^2 dv_\phi\end{split}
\end{equation}
We also estimate 
\begin{equation}\label{p10}
C_7+\frac{\l C\Delta \phi}{\l-1}\leq \l(C_1+C)(n+\Delta \phi)
\end{equation}
Putting \eqref{p07}, \eqref{p08}, \eqref{p09} and \eqref{p10} together, we have
\begin{equation}\label{p11}
\begin{split}
\int_M \left((p-1)- \frac{(p-1)^2}{4(\l-1)} \right)&u^{p-2}|\nabla_\phi u|_\phi^2 dv_\phi+\int_{M}\frac{\l C}{4}\Lambda_\phi\omega^T u^p dv_\phi\\
&\leq \int_M \l(C_1+C)u^pdv_\phi
\end{split}
\end{equation}
Note that $\Lambda_\phi\omega^T\geq (n+\Delta\phi)^{\frac{1}{n-1}}e^{\frac{-F}{n-1}}$. It follows from \eqref{p11} (taking $\l=p$) that, 
\begin{equation}\label{p12}
\int_M (n+\Delta \phi)^{p+\frac{1}{n-1}}dv_g\leq C_8(p, g, |F|, |\phi|) \int_M (n+\Delta \phi)^p dv_g. 
\end{equation}
By iterating, we have the following, for any $p\geq 1$, 
\begin{equation}
\int_M (n+\Delta\phi)^pdv_g\leq C(p, C_0)
\end{equation}

\end{proof}

\subsection{Estimate of $\nabla F$}
In this section we estimate $|\nabla_\phi F|_\phi$ and $|\nabla F|$. The key is the following,

\begin{thm}\label{pmain1}There exists a constant $C_0$ such that
\begin{equation}
|\nabla_\phi F|_\phi\leq C_0.
\end{equation}
\end{thm}

\begin{proof}We compute \begin{equation}\label{first01}
\begin{split}
\Delta_\phi\left(e^{\frac{F}{2}}|\nabla_\phi F|^2_\phi\right)=&\Delta_\phi(e^{\frac{F}{2}}) |\nabla_\phi F|_\phi^2+e^{\frac{F}{2}}\Delta_\phi(|\nabla_\phi F|^2_\phi)+e^{\frac{F}{2}}\nabla_\phi F\cdot_\phi \nabla_\phi(|\nabla_\phi F|_\phi^2)\\
=&e^{\frac{F}{2}}(\frac{\Delta_\phi F}{2}|\nabla_\phi F|_\phi^2+\frac{|\nabla_\phi F|_\phi^4}{4}+\Delta_\phi(|\nabla_\phi F|^2_\phi)+\nabla_\phi F\cdot_\phi \nabla_\phi(|\nabla_\phi F|_\phi^2))
\end{split}
\end{equation}
By Bochner's formula, we have
\begin{equation}\label{first02}
\Delta_\phi(|\nabla_\phi F|^2_\phi)=g^{i\bar j}_\phi g^{k\bar l}_\phi (F_{, ki}F_{, \bar j \bar l}+F_{k\bar j} F_{i\bar l})+2\nabla_\phi F\cdot_\phi \nabla_\phi\Delta_\phi F+g^{i\bar j}_\phi g^{k\bar l}_\phi Ric_{\phi, i\bar l}F_kF_{\bar j}
\end{equation}
We compute 
\begin{equation}\label{first03}
\begin{split}
\nabla_\phi F\cdot_\phi \nabla_\phi(|\nabla_\phi F|^2_\phi)=&\frac{1}{2}g^{i\bar j}_\phi\left(F_{\bar j} (|\nabla_\phi F|^2_\phi)_{i}+F_{i}(|\nabla_\phi F|^2_\phi)_{\bar j}\right)\\
=&\text{Re}\left( g^{i\bar j}_\phi g^{k\bar l}_\phi F_i(F_{ k\bar j}F_{\bar l}+F_kF_{, \bar j\bar l})\right)
\end{split}
\end{equation}
In \eqref{first02} and \eqref{first03}, $F_{, ik}$ and $F_{, \bar j \bar l}$ denote covariant derivatives w.r.t  $\omega_\phi$. 
Observe that 
\begin{equation}\label{first04}
\frac{|\nabla_\phi F|_\phi^4}{4}+g^{i\bar j}_\phi g^{k\bar l}_\phi F_{, ki}F_{, \bar j \bar l}+\text{Re}\left( g^{i\bar j}_\phi g^{k\bar l}_\phi F_iF_kF_{, \bar j\bar l}\right)\geq 0.
\end{equation}
and also observe that, using $Ric_{\phi, i\bar l}+ F_{i\bar l}=Ric_{i\bar l}$,
\begin{equation}\label{first05}
\begin{split}
g^{i\bar j}_\phi g^{k\bar l}_\phi Ric_{\phi, i\bar l}F_kF_{\bar j}
+\text{Re}\left( g^{i\bar j}_\phi g^{k\bar l}_\phi F_iF_{ k\bar j}F_{\bar l}\right)=&g^{i\bar j}_\phi g^{k\bar l}_\phi Ric_{\phi, i\bar l}F_kF_{\bar j}
+\text{Re}\left( g^{i\bar j}_\phi g^{k\bar l}_\phi F_kF_{ i\bar l}F_{\bar j}\right)\\
=&g^{i\bar j}_\phi g^{k\bar l}_\phi Ric_{i\bar l}F_kF_{\bar j}
\end{split}
\end{equation}
Putting \eqref{first01}, \eqref{first02}, \eqref{first03} and \eqref{first04} together, we get
\begin{equation}\label{first06}
\Delta_\phi\left(e^{\frac{F}{2}}|\nabla_\phi F|^2_\phi\right)\geq e^{\frac{F}{2}}\left(\frac{\Delta_\phi F}{2}|\nabla_\phi F|_\phi^2+g^{i\bar j}_\phi g^{k\bar l}_\phi Ric_{i\bar l}F_kF_{\bar j}+2\nabla_\phi F\cdot_\phi \nabla_\phi \Delta_\phi F\right)
\end{equation}
Note that by the equation \eqref{scalar2},
\begin{equation}\label{first07}
\Delta_\phi F=\text{tr}_\phi \gamma-h\geq -C(\Lambda_\phi\omega^T+1)\geq -C\left(e^{-F}(n+\Delta \phi)^{n-1}+1\right)\geq -C_1(n+\Delta \phi)^{n-1},
\end{equation}
where the last step we have used the estimate $|F|\leq C_0$ and $n+\Delta \phi\geq e^{\frac{F}{n}}$. 
We also estimate
\begin{equation}\label{first08}
g^{i\bar j}_\phi g^{k\bar l}_\phi Ric_{i\bar l}F_kF_{\bar j}\geq -C\Lambda_\phi \omega^T |\nabla_\phi F|_\phi^2\geq -C_1 (n+\Delta\phi)^{n-1} |\nabla_\phi F|_\phi^2.
\end{equation}
Putting \eqref{first07} and \eqref{first08} back to \eqref{first06}, we estimate
\begin{equation*}
\Delta_\phi\left(e^{\frac{F}{2}}|\nabla_\phi F|^2_\phi\right)\geq -C_2 e^{\frac{F}{2}} |\nabla_\phi F|_\phi^2 (n+\Delta\phi)^{n-1}+2e^{\frac{F}{2}}\nabla_\phi F\cdot_\phi \nabla_\phi \Delta_\phi F
\end{equation*}
Denote $v=e^{\frac{F}{2}}|\nabla_\phi F|^2_\phi$, we rewrite the above as
\begin{equation}\label{first09}
\Delta_\phi v\geq -C_2 v (n+\Delta \phi)^{n-1}+2e^{\frac{F}{2}}\nabla_\phi F\cdot_\phi \nabla_\phi \Delta_\phi F
\end{equation}
We start the integral estimates, for $p>1$
\begin{equation}\label{first10}
\begin{split}
\int_M (p-1)v^{p-2}|\nabla_\phi v|_\phi^2 dv_\phi=&\int_M v^{p-1}(-\Delta_\phi v) dv_\phi\\
\leq & C_2\int_M v^pG dv_\phi-2\int_M v^{p-1} e^{\frac{F}{2}}\nabla_\phi F\cdot_\phi \nabla_\phi \Delta_\phi Fdv_\phi,
\end{split}
\end{equation}
where we take $G=(n+\Delta\phi)^{2n-2}+1$.
Using integration by parts to treat \eqref{first10} we claim
\begin{equation}\label{first11}
\int_M v^{p-2}|\nabla_\phi v|^2_\phi dv_\phi \leq \frac{pC_3}{p-1}\int_M v^p G dv_g,
\end{equation}
We compute, 
\begin{equation}\label{first12}
\begin{split}
-2\int_M v^{p-1} e^{\frac{F}{2}}\nabla_\phi F\cdot_\phi \nabla_\phi \Delta_\phi Fdv_\phi=&2\int_M \nabla_\phi( v^{p-1} e^{\frac{F}{2}}\nabla_\phi F)\Delta_\phi F dv_\phi\\
=&2\int_Mv^{p-1}e^{\frac{F}{2}}\left((\Delta_\phi F)^2+|\nabla_\phi F|^2_\phi\Delta_\phi F\right)dv_\phi\\
&\; +2\int_M (p-1)v^{p-2}\nabla_\phi v e^{\frac{F}{2}}\nabla_\phi F\Delta_\phi F dv_\phi
\end{split}
\end{equation}
The first term in the righthand side of \eqref{first12} can be controlled in a straightforward way,
\begin{equation}\label{first13}
\int_Mv^{p-1}e^{\frac{F}{2}}\left((\Delta_\phi F)^2+|\nabla_\phi F|^2_\phi\Delta_\phi F\right)dv_\phi\leq C_4 \int_M v^p Gdv_g,
\end{equation}
noting that $e^{\frac{F}{2}}|\nabla_\phi F|_\phi^2\Delta_\phi F=(v-1)\Delta_\phi F\leq v|\Delta_\phi F|$ and $(\Delta_\phi F)^2\leq C_4G$.
We compute
\begin{equation}\label{first14}
\begin{split}
\int_M v^{p-2}\nabla_\phi v e^{\frac{F}{2}}\nabla_\phi F\Delta_\phi F dv_\phi\leq& \frac{1}{4}\int_M v^{p-2}|\nabla_\phi v|^2 dv_\phi+\int_M v^{p-2} e^F |\nabla_\phi F|^2(\Delta_\phi F)^2 dv_\phi\\
\leq & \frac{1}{4}\int_M v^{p-2}|\nabla_\phi v|^2 dv_\phi+C_4\int_M v^{p-1} G  dv_g
\end{split}
\end{equation}
The claim \eqref{first11} is a direct consequence of \eqref{first12}, \eqref{first13} and \eqref{first14}. We rewrite \eqref{first11}
\begin{equation}\label{first15}
\int_M |\nabla_\phi v^{\frac{p}{2}}|_\phi^2 dv_g\leq \frac{C_3p^3}{p-1}\int_M v^{p}G dv_g
\end{equation}
Using $|\nabla f|_g^2\leq |\nabla_\phi f|^2_\phi (n+\Delta \phi)$, we compute
\begin{equation}\label{first16}
\begin{split}
\int_M |\nabla (v^{\frac{p}{2}})|^{2-2\epsilon}dv_g\leq& \int_M |\nabla_\phi v^{\frac{p}{2}}|_\phi^{2-2\epsilon}(n+\Delta \phi)^{1-\epsilon} dv_g\\
\leq &\left(\int_M |\nabla_\phi v^{\frac{p}{2}}|_\phi^2 dv_g\right)^{1-\epsilon}\left(\int_M (n+\Delta \phi)^{{\frac{1}{\epsilon}-1}}dv_g\right)^\epsilon
\end{split}
\end{equation}
That is, assuming $p$ is bounded away from $1$
\begin{equation}\label{first17}
\|\nabla (v^{\frac{p}{2}})\|_{L^{2-2\epsilon}}\leq C_5p\left(\int_M v^{p}G dv_g\right)^{\frac{1}{2}} \|n+\Delta \phi\|_{L^{\frac{1}{\epsilon}-1}}.
\end{equation}
Applying Sobolev inequality, we have
\begin{equation}\label{first18}
\|v^{\frac{p}{2}}\|_{L^{\frac{2(1-\epsilon)(2n+1)}{2n-1+2\epsilon}}}\leq c_0 \left(\|\nabla (v^{\frac{p}{2}})\|_{L^{2(1-\epsilon)}}+\|v^{\frac{p}{2}}\|_{L^{2(1-\epsilon)}}\right)
\end{equation}
We compute, using \eqref{first17} and \eqref{first18} and denoting $K_\epsilon=\|n+\Delta \phi\|_{L^{\frac{1}{\epsilon}-1}}$, 
\begin{equation}\label{first19}
\|v^{\frac{p}{2}}\|_{L^{\frac{2(1-\epsilon)(2n+1)}{2n-1+2\epsilon}}}\leq C_6 K_\epsilon p \left(\int_M v^{p}G dv_g\right)^{\frac{1}{2}}+c_0\|v^{\frac{p}{2}}\|_{L^{2(1-\epsilon)}},
\end{equation}
We compute
\begin{equation}\label{first20}
\left(\int_M v^{p}G dv_g\right)^{\frac{1}{2}}\leq \|v^{\frac{p}{2}}\|_{L^{\frac{2}{1-\epsilon}}} \sqrt{\|G\|_{L^{\frac{1}{\epsilon}}}}
\end{equation}
It follows that, with $L_\epsilon=\sqrt{\|G\|_{L^{\frac{1}{\epsilon}}}}$, 
\begin{equation}\label{first21}
\|v^{\frac{p}{2}}\|_{L^{\frac{2(1-\epsilon)(2n+1)}{2n-1+2\epsilon}}}\leq C_7 (p K_\epsilon L_\epsilon+1) \|v^{\frac{p}{2}}\|_{L^{\frac{2}{1-\epsilon}}}
\end{equation}
Choose $\epsilon$ sufficiently small ( say $\epsilon=(2n+2)^{-1}$) such that
\[
\frac{2(1-\epsilon)(2n+1)}{2n-1+2\epsilon}>\frac{2}{1-\epsilon}.
\]
By Theorem  \ref{pmain}, $K_\epsilon, L_\epsilon$ are uniformly bounded above. We can then get, for some uniformly positive constant $C_8\geq 2$,
\begin{equation}\label{first22}
\|v^{\frac{p}{2}}\|_{L^{\frac{2(1-\epsilon)(2n+1)}{2n-1+2\epsilon}}}\leq C_8 p \|v^{\frac{p}{2}}\|_{L^{\frac{2}{1-\epsilon}}},
\end{equation}
A standard iteration procedure then implies that
\begin{equation}
\|v\|_{L^\infty}\leq C_9 \|v\|_{L^1}. 
\end{equation}
To bound the $L^1$ norm, we compute
\begin{equation}\label{first23}
\Delta_\phi e^{\frac{F}{2}}=\frac{e^{\frac{F}{2}}}{4}|\nabla_\phi F|^2_\phi+\frac{e^{\frac{F}{2}}}{2}\Delta_\phi F. 
\end{equation}
Hence we have
\begin{equation}\label{first24}
\begin{split}
\int_M vdv_g=&\int_M e^{\frac{F}{2}}|\nabla_\phi F|^2_\phi dv_g\leq C_{10} \int_M e^{\frac{F}{2}}|\nabla_\phi F|^2_\phi dv_\phi=2C_{10}\int_M e^{\frac{F}{2}}(-\Delta_\phi F) dv_\phi\\
\leq &C_{11}\int_M (\Lambda_\phi \omega^T+1)dv_\phi=(n+1)C_{11}\text{Vol}(M). 
\end{split}
\end{equation}
This completes the proof. 
\end{proof}
As a direct consequence, we have the following,
\begin{cor}
For any $p\geq 1$, there exists a constant $C_p$ such that
\[
|\nabla F|_{L^p}\leq C_p=C(C_0, p).
\]
\end{cor}
\begin{proof}
We have the following,
\[
|\nabla F|^2\leq |\nabla_\phi F|^2_\phi (n+\Delta \phi). 
\]
It then follows from Theorem \ref{pmain} and Theorem \ref{pmain1}. 
\end{proof}

\subsection{Estimate on $(n+\Delta\phi)$ and higher order estimates}
We prove the following
\begin{thm}\label{2nd}There exists a uniformly bounded constant $1<C_0$ such that
\begin{equation}
C_0^{-1}\leq n+\Delta \phi\leq C_0, \|\phi\|_{C^k}\leq C=C(k, C_0).
\end{equation}
\end{thm}
We need to prove that $n+\Delta\phi$ is bounded above given $F\in W^{1, p}$, for $p$ sufficiently large. Higher order estimates follow from standard elliptic bootstrapping theory given the estimates of $n+\Delta \phi$, and the H\"older estimate \cite{DZZ, wangyu}. 
Indeed a direct adaption of \cite{chenhe12} gives the corresponding Sasaki setting,
\begin{thm}\label{second}
Let $\phi$ be a basic smooth function on a compact Sasaki manifold $M^{2n+1}$ such that $\omega^T+\sqrt{-1}\p_B\bar \p_B\phi>0$. Suppose $F$ is a smooth function with a uniformly $W^{1, p}$ bound, $p>2n+1$ such that
\[
\eta\wedge (\omega^T+\sqrt{-1}\p_B\bar\p_B\phi)^n=e^F\eta\wedge (\omega^T)^n, \sup \phi=0
\]
then there exists a uniformly bounded constant $1<C_0=C_0(\|F\|_{W^{1, p}}, p, M, g)$ such that
\[
C_0^{-1}\leq n+\Delta \phi\leq C_0
\]
\end{thm}
\begin{proof}
The computation and the argument are almost identical, except when applying the Sobolev inequality. For that one needs to replace the dimensional constant $2n$ by $2n+1$. Note that we have already proved that $\phi$ has $C^0$ bound in Sasaki setting, using Alexanderov maximum principle (see Lemma \ref{cma0}). This would be sufficient to prove Theorem \ref{second}. We shall keep the discussions brief.

The computation in \cite{chenhe12}[Section 2] applies directly to tranverse K\"ahler structure here. First we follow Yau's computation \cite{Yau} to get, with $u=\exp(-C_1\phi)(n+\Delta\phi)$,
\begin{equation}
\Delta_\phi (u)\geq C_2(n+\Delta\phi)^{\frac{n}{n-1}}+\exp(-C_1\phi F)-C_3,
\end{equation}
where $C_1$ is sufficiently  positive constant  such that $C_1+\inf R^T_{i\bar i l\bar l}\geq 1$ and $C_2, C_3$ are positive constants depending on $C_1, \|\phi\|_{L^\infty}, \sup F, n$ and the transverse curvature of $(M, g)$. Proceeding exactly the same as in K\"ahler \cite{chenhe12}[see (2.12)], we get for $p\geq 1/4$
\begin{equation}
\int_M \left(|\nabla (u^p)|^2+pC_7 u^{2p+\frac{n}{n-1}}\right)dv_g\leq p^2 C_8\int_M u^{2p}\left(|\nabla F|^2+1\right)dv_g. 
\end{equation}
where $C_7=C_7(\|\phi\|_{L^\infty}, \|F\|_{L^\infty}, n)$ and $C_8=C_8(\|\phi\|_{L^\infty}, \|\nabla \phi\|_{L^\infty}, n)$. We use the Sobolev inequality for $(M, g)$,
\begin{equation}
\left(\int_M |f|^{\frac{2(2n+1)}{2n-1}}\right)^{\frac{2n-1}{2n+1}}\leq C\left(\int_M |\nabla f|^2dv_g+\text{Vol(M)}^{-1}\int_M f^2 dv_g\right)
\end{equation}
By taking $f=u^p$, we have for $q=\frac{2(2n+1)p}{2n-1},$
\begin{equation}
\|u\|_{L^q}\leq (pC)^{1/p}\left(\int_M u^{2p}(|\nabla F|^2+1)dv_g\right)^{\frac{1}{2p}}
\end{equation}
Applying H\"older inequality with $q_0^{-1}+2p_0^{-1}=1$, we get
\begin{equation}
\int_M u^{2p}(|\nabla F|^2+1)dv_g\leq \left(\int_M u^{2pq_0}\right)^{\frac{1}{q_0}}\left(\int_M (|\nabla F|^2+1)^{\frac{p_0}{2}}\right)^{\frac{2}{p_0}}
\end{equation}
Since we assume $|\nabla F|\in L^{p_0}$ for some $p_0>2n+1$, we have
\begin{equation}
\|u\|_{L^q}\leq (pC)^{1/p}\|u\|_{L^{2pq_0}}. 
\end{equation}
Since $q=\frac{2(2n+1)p}{2n-1}>2pq_0$ for $p_0>2n+1$, the iteration process implies that $\|u\|_{L^\infty}\leq C\|u\|_{L^1}\leq C$, where the constant depends on 
\[
C=C(\|\phi\|_{L^\infty}, \|\nabla \phi\|_{L^\infty}, \|F\|_{W^{1, p_0}}, p_0, M, g, n). 
\]
We can clearly follow the procedure as in \cite{chenhe12}[Section 3] to prove the bound on $\|\nabla \phi\|_{L^\infty}$.  We shall skip the details.  

\end{proof}

\begin{rmk}
For our application, on the other hand, we already know that $\Delta \phi\in L^p$ for $p$ sufficiently large  in this setting.  Hence $\|\nabla \phi\|_{L^\infty}<\infty$ holds for free.
\end{rmk}

Another way to derive the estimate of $n+\Delta \phi$ can be done for complex Monge-Ampere, given sufficiently high $L^p$ norm of $(n+\Delta \phi)$ and $W^{1, q}$ norm of $F$. This problem can be treated purely locally. 
We consider a smooth plurisubharmonic function $u$, defined on a unit ball $B_1\subset \C^{n}$, such that for $F\in W^{1, q}$, $q>2n$,
\[
\log \det(u_{i\bar j})=F
\]
Then we have the following, 
\begin{prop}\label{cma4}
There exists $p$ sufficiently large depending on $q>2n$, such that if $\|\Delta u\|_L^p(B_1)<\infty$, then there exists a uniform constant $C>0$,
\[
\sup_{B_{1/2}}\Delta u\leq C,
\]
where $C=C(p, \|F\|_{W^{1, q}(B_1)}, \|\Delta u\|_{L^p(B_1)})$
\end{prop}

In \cite{he10}, when $q=\infty$, the author proved the above interior estimates using integral methods, for $p>n^2$. 
The proof in \cite{he10} can be directly adapted to to prove Proposition \ref{cma4}. 
Nevertheless, it is relatively well-known now that one can obtain the interior estimate of $\Delta u$, in terms of $\|\Delta u\|_{L^p}, \|F\|_{W^{1, q}}$ given $p, q$ are sufficiently large (or even this holds  for $F\in C^\alpha, u\in C^{1, \beta}$ see \cite{LLZ}), we shall skip the details of the proof of Proposition \ref{cma4}. 

\section{Scalar curvature type equation with more flexible righthand side}
We consider the equation
\begin{equation}\label{scalar3}
\begin{split}
&\det(g_{i\bar j}^T+\phi_{i\bar j})=e^F\det(g_{i\bar j}^T)\\
&\Delta_\phi F=\text{tr}_\phi (Ric-\beta)-h,
\end{split}
\end{equation} 
where $\beta=\beta_0+\sqrt{-1}\p_B\bar\p_B f$, such that $\beta_0$ is a uniformly bounded real $(1, 1)$-form and $f$ is a basic function, with $\sup f=0$ satisfying 
\begin{equation}\label{assumption01}
\beta=\beta_0+\sqrt{-1}\p_B\bar\p_B f \geq 0, \int_M e^{-p_0f} dv_g\leq c_0<\infty
\end{equation}
We rewrite the equation as follows,
\begin{equation}\label{scalar4}
\begin{split}
&\det(g_{i\bar j}^T+\phi_{i\bar j})=e^F\det(g_{i\bar j}^T)\\
&\Delta_\phi (F+f)=\text{tr}_\phi \gamma-h, \gamma=Ric-\beta_0
\end{split}
\end{equation} 
We will prove the following,
\begin{thm}\label{main1}Assume \eqref{assumption01}. Let $\phi$ be a smooth solution of \eqref{scalar4}. There exists a positive constant $c_n$ depending only on $n$ and we assume $p_0>c_n+2$. Then for any $p<p_0$, there exists a uniformly bounded constant $C$, such that
\begin{equation}
\|F+f\|_{W^{1, 2p}}\leq C, \|n+\Delta \phi\|_{L^p(M, g)}\leq C=C\end{equation}
where $C$ depends on $C_0, p, p_0$ and $\int_M e^{-p_0f}dv_g$. The bounds are uniform in $p_0$ and $p$, provided that $p$ bounded away from $p_0$ (say $p\leq p_0-1$). 
\end{thm}

This corresponds to \cite{CC3}[Theorem 2.3]. Our proof of Theorem \ref{main0} follows indeed largely the proof in \cite{CC3}[Theorem 2.3] and Chen-Cheng's proof is designed for this general situation. 
The proof proceeds exactly the same as in \cite{CC3}[Theorem 2.3] and there are some necessary modifications, in particular for $C^0$ estimates of $F$ and $\phi$, adapted to Sasaki setting (see Theorem \ref{main0}). 
The proof with a more flexible term $f$ is parallel to Theorem \ref{main0}. 
We present the arguments briefly.

\subsection{The bound of $|F+f|$ and $|\phi|$}
The key is to prove the following, 
\begin{lemma}\label{cma4}
Let $\psi$ be the solution of the following equation with $\sup\psi=0$,
\begin{equation}\label{cma5}
\det(g_{i\bar j}^T+\psi_{i\bar j})= \underline Ce^F \sqrt{F^2+1}\det(g_{i\bar j}^T),
\end{equation}
where the constant $\underline C$ is given by, 
\[
\underline C= A^{-1}\int_M dv_g, A:=\int_M e^F \sqrt{F^2+1}\;dv_g
\]
Then for any $\epsilon_0\in (0, 1)$, there exist constant $C, C_1$ such that,
\begin{equation}
F+f+\epsilon_0 \psi-C\phi\leq C_1,
\end{equation}
where we have $C=2(\max |\gamma|_{\omega^T}+1)$ and $C_1=C_1(\epsilon_0, C_0)$. 
\end{lemma}

\begin{proof}
The proof is almost identical to Lemma \ref{cma1}, replacing $F$  by $F+f$. Denote $u=F+f+\epsilon \psi-C\phi$. Then we have as in \eqref{u01},
\begin{equation}
\Delta_\phi u\geq  \frac{C}{2} \Lambda_\phi \omega^T+\epsilon_0A^{-\frac{1}{n}} (F^2+1)^{\frac{1}{2n}}-nC
\end{equation}
We construct the same auxiliary function $\zeta_p$, then the arguments apply exactly as in Lemma \ref{cma1}, noticing $f, \psi\leq 0$ (this is used from \eqref{c08} to \eqref{c09}). 
\end{proof}

Next we have the following
\begin{lemma}For $p_0\geq 4$, there exists a constant $C_1$ such that
\begin{equation}
F+f\leq C_1(C_0, p_0), -\phi\leq C_1=C_1(C_0, p_0).
\end{equation}
\end{lemma}
\begin{proof}
Proceeding as in Lemma \ref{upper01} (using $\alpha$-invariant for $\psi$), we have for any $p>1$
\begin{equation}
\int_M e^{p(F+f)}dv_g\leq C_2=C_2(p, C_0),
\end{equation}
where $C_2$ depends uniformly on $p$ (to be precise, in the order of $e^p$).
We compute
\begin{equation}
\int_M e^{3F}dv_g\leq \int_M e^{3(F+f)} e^{-3f} dv_g\leq \left(\int_M e^{3(F+f)(1+\epsilon)/\epsilon}dv_g\right)^{\frac{\epsilon}{1+\epsilon}} \left(\int_M e^{-3(1+\epsilon) f}dv_g\right)^{\frac{1}{1+\epsilon}}
\end{equation}
Take $\epsilon=1/3$, then we get that $\int_M e^{3F}dv_g\leq C$. By Lemma \ref{cma0} we have $-\phi \leq C$. For the estimate on $|\psi|$, we have the following,
\[
\det{(g_{i\bar j}^T+\psi_{i\bar j})}= e^{\tilde F}\det(g_{i\bar j}^T)
\]
Then $e^{\tilde F}=\underline C e^{F}\sqrt{F^2+1}$. Then we compute, for $\epsilon=1/4$
\[e^{(2+2\epsilon)\tilde F}=(\underline C)^{(2+2\epsilon)} e^{(2+2\epsilon)F}(F^2+1)^{1+\epsilon}.\]
Note that $(\underline C)^2$ is a bounded constant and $e^{(2+2\epsilon)F}(F^2+1)^{1+\epsilon}\leq C(e^{3F}+1)$ for a uniformly bounded constant $C$. We can hence apply Lemma \ref{cma0} again to get $-\psi \leq C$.
This completes the proof. 
\end{proof}

\begin{rmk}The requirement of $p_0\geq 4$ can be replaced by $p_0>2$. However since our estimate in Lemma \ref{cma0} is weaker than Kolodziej's $C^0$ estimate \cite{K} in K\"ahler setting, we do need $p_0>2$ instead of $p_0>1$ (as in \cite{CC3}[Corollary 2.2]). It would be an interesting question to extend Kolodziej's $C^0$ estimate to Sasaki setting.  
\end{rmk}

We would like to estimate the lower bound of $F+f$. Since $f$ might not be bounded below, the pointwise maximum principle as in Lemma \ref{lower001} does not apply. Instead we use Alexanderov maximum principle as in \cite{CC3}[Lemma 2.3]

\begin{lemma}\label{lower002} There exists a uniformly bounded constant $C_5$ such that
\[
F+f\geq -C_5. 
\]
In particular $F\geq -C_5$ since $f\leq 0$. 
\end{lemma}

\begin{proof}
This is parallel to Lemma \ref{cma1} but easier. 
Choose  $C$ such that for basic real $(1, 1)$-form $\gamma$, $\gamma\leq (C-1)\omega^T$.
We compute, \begin{equation}\label{low01}
\Delta_\phi (F+f+C\phi)=\Lambda_\phi \gamma-h+Cn-C\Lambda_\phi \omega^T\leq |h|+Cn-\Lambda_\phi \omega^T
\end{equation}
Let the auxiliary function $\zeta_p$ be defined in Lemma \ref{cma1}. Fix a constant $\delta_1$. Denote $u_1=F+f+C\phi$. Suppose $u_1$ takes its minimum at $p_1\in M$. We compute,
\begin{equation}\label{low02}
\Delta_{\eta_\phi}(e^{-\delta_1 u_1}\zeta_{p_1})=\Delta_{\eta_\phi}(e^{-\delta_1 u_1}) \zeta_{p_1}+e^{-\delta_1 u_1}\Delta_{\eta_\phi}(\zeta_{p_1})+2\nabla_{\eta_\phi} (e^{-\delta_1 u_1})\cdot \nabla_{\eta_\phi}\zeta_{p_1}.
\end{equation}
Since $u_1$ is a basic function, we have
\begin{equation}\label{low03}
\Delta_{\eta_\phi}(e^{-\delta_1 u_1})=\Delta_\phi(e^{-\delta_1 u_1})=e^{-\delta_1 u} \delta_1^2 |\nabla_\phi u_1|_\phi^2-e^{-\delta_1 u_1}\delta_1 \Delta_\phi u_1
\end{equation}
We compute
\begin{equation}\label{low04}
\Delta_{\eta_\phi}(\zeta_{p_1})=\nabla^2_{\xi, \xi}(\zeta_{p_1})+\Delta_\phi(\zeta_{p_1})\geq -(1+\Lambda_\phi \omega^T)|\nabla^2\zeta_{p_1}|.
\end{equation}
We also compute, noting that $u_1$ is basic,
\begin{equation}\label{low05}
\begin{split}
2\nabla_{\eta_\phi} (e^{-\delta_1 u_1})\cdot \nabla_{\eta_\phi}\zeta_{p_1}=&-2\delta_1 e^{-\delta_1 u_1}\nabla_\phi u_1\cdot_\phi \nabla_\phi \zeta_{p_1}\\
\geq& -\delta_1^2 e^{-\delta_1}|\nabla_\phi u_1|^2_\phi\zeta_{p_1}-e^{-\delta_1 u_1}\Lambda_\phi\omega^T |\nabla \zeta_{p_1}|^2 \zeta_{p_1}^{-1}
\end{split}
\end{equation}
Combining \eqref{low01}, \eqref{low02}, \eqref{low03}, \eqref{low04} and \eqref{low05}, we have
\begin{equation}\label{low06}
\begin{split}
\Delta_{\eta_\phi}(e^{-\delta_1 u_1})\geq& e^{-\delta_1 u_1}\Lambda_\phi \omega^T(\delta_1\zeta_{p_1}- |\nabla \zeta_{p_1}|^2 \zeta_{p_1}^{-1}-|\nabla^2\zeta_{p_1}|)\\
&-C_1\delta_1 e^{-\delta_1 u_1}\zeta_{p_1}-e^{-\delta_1 u_1}|\nabla^2\zeta_{p_1}|
\end{split}
\end{equation}
 We choose $\theta$ (as in Lemma \ref{cma1}) sufficiently small, such that
 \[
 \delta_1\zeta_{p_1}- |\nabla \zeta_{p_1}|^2 \zeta_{p_1}^{-1}-|\nabla^2\zeta_{p_1}|\geq \delta_1(1-\theta)-2\theta r^{-2}-4\theta^2 r^{-2}(1-\theta)^{-1}>0
 \]
 Hence by \eqref{low06}, we have
 \begin{equation}\label{low07}
 \Delta_{\eta_\phi}(e^{-\delta_1 u_1})\geq -(C_1+1)\delta_1 e^{-\delta_1 u_1}\zeta_{p_1}.
 \end{equation}
Now we apply the Alexanderov maximum principle (Lemma \ref{alexanderov}) in $B_r(p_1)$, with  $\delta_1=1/(2n+1)$,
\begin{equation}\label{low08}
\begin{split}
e^{-\delta_1u_1(p_1)}\leq& \sup_{\p B} e^{-\delta_1 u_1}\zeta_{p_1}+C_n r\left(\int_Me^{2F}e^{-(2n+1)\delta_1u_1}(C_1+1)^{2n+1}\delta_1^{2n+1}dv_g\right)^{\frac{1}{2n+1}}\\
\leq & \sup_{\p B} e^{-\delta_1 u_1}\zeta_{p_1}+C_9\left(\int_M e^{F+f- 2f}dv_g\right)^{\frac{1}{2n+1}} \\
\leq &  \sup_{\p B} e^{-\delta_1 u_1}\zeta_{p_1}+C_{10}\left(\int_M e^{-2f}dv_g\right)^{\frac{1}{2n+1}},
\end{split}
\end{equation}
where we have used the fact that $F+f\leq C_1$. If $p_0\geq 2$, this implies that $u_1(p_1)\geq C_{11}$. 

\end{proof}

\begin{rmk}The $C^0$ estimate does not require the lower bound of $\sqrt{-1}\p_B\bar\p_B f$. 
\end{rmk}

\subsection{ $L^p$ estimate of $(n+\Delta\phi)$ and  the bound $|\nabla_\phi (F+f)|_\phi$}
The estimates of this section correspond to \cite{CC3}[Section 2.2] and these estimates are very important for applications in existence of csck when the automorphism group is not discrete as in \cite{CC3}. In Sasaki setting these estimates hold with almost identical arguments, given the $C^0$ estimates obtained above. We have already proved similar estimates using the method in  Theorem \ref{pmain} and Theorem \ref{pmain1} (with the absence of $f$). We shall emphasize that since in Sasaki setting, all the quantities involved in the arguments are basic, and all the computations work the same way as in K\"ahler via its transverse K\"ahler structure. The only difference is the application of Sobolev inequality, where the dimension is different. Hence the results and arguments in \cite{CC3}[Section 2.2] work almost identical here, with modifications as indicated in Theorem \ref{pmain1}. We keep the discussion brief. 
\begin{lemma}\label{pp01}For any $p\geq 1$, there exists a constant $C_{12}$
\begin{equation}
\int_{M}e^{(p-1)f}(n+\Delta \phi)^p dv_g\leq C_{12},
\end{equation}
where $C_{12}$ depends on $C_0$, $p$ and the upper bound of $\int_M e^{-p_0f}dv_g$. 
\end{lemma}

\begin{proof}This corresponds to \cite{CC3}[Theorem 2.1] and its method is used to prove Theorem \ref{pmain} above. We shall compute, for  constants $\delta\in (0, 1), \kappa>1, C>0$ as in \cite{CC3}[Section 2.2], 
\begin{equation}
\Delta_\phi (e^{-\kappa(F+\delta f+C\phi)}(n+\Delta \phi)).
\end{equation}
Since all quantities involved are basic functions, hence the computation is identical to the K\"ahler case via its transverse K\"ahler structure. The proof works word by word  as in \cite{CC3}[Theorem 2.1]. We skip the details. 
\end{proof}

\begin{cor}For any $1\leq q\leq p_0-1$, there exists a constant $C_q$ such that
\[
\int_M (n+\Delta \phi)^q dv_g\leq C_q. 
\]
\end{cor}

\begin{proof}This is a direct consequence of Lemma \ref{pp01} as in \cite{CC3}[Corollary 2.4]. Take $s=(q-1)p_0/(p_0-1)\geq 0$. We compute
\begin{equation}
\begin{split}
\int_M (n+\Delta \phi)^q dv_g=&\int_M e^{-sf} e^{sf}(n+\Delta\phi)^q dv_g\\
\leq& \left(\int_M e^{-p_0f}\right)^{\frac{s}{p_0}} \left(\int_M e^{\frac{sp_0}{p_0-s}f}(n+\Delta \phi)^{\frac{p_0q}{p_0-s}}\right)^{1-\frac{s}{p_0}}
\end{split}
\end{equation}
\end{proof}

Next we estimate the bound $|\nabla_\phi w|_\phi^2$ for $w=F+f$. 
\begin{lemma}there exists $c_n>0$, depending only on $n$, such that if $p_0>c_n+2$, we have
\[
|\nabla_\phi w|_\phi \leq C_{14}, 
\]
where $C_{14}$ depends only on $C_0$ and the bound $\int_M e^{-p_0f}dv_g$.\end{lemma}

\begin{proof}The proof is almost identical to \cite{CC3}[Theorem 2.2] and Theorem \ref{pmain1} above. Denote $u=e^\frac{w}{2}|\nabla_\phi w|_\phi^2$. Then exactly as in \cite{CC3}[Section 2.2], we have the lower bound
\begin{equation}
\Delta_\phi u\geq 2e^\frac{w}{2}\nabla_\phi w\cdot_\phi \nabla_\phi \Delta_\phi w-u\tilde G,
\end{equation}
where one can actually take $\tilde G=C_{15}((n+\Delta \phi)^{n-1}+1)$. We have $F\geq -C_5$ as in Lemma \ref{low02}. As in \eqref{first08}, we have
\begin{equation}
g^{i\bar j}_\phi g^{k\bar l}_\phi Ric_{i\bar l}w_kw_{\bar j}\geq -C\Lambda_\phi \omega^T |\nabla_\phi w|_\phi^2\geq -C_1 (n+\Delta\phi)^{n-1} |\nabla_\phi w|_\phi^2.
\end{equation}
We then start integral estimate, as in \cite{CC3}[Theorem 2.2] and Theorem \ref{pmain1} above. We have
\begin{equation}\label{gra01}
\begin{split}
\int_M (p-1)u^{p-2}&|\nabla_\phi u|^2_\phi dv_\phi=\int_M u^{p-1}(-\Delta_\phi u)\\
\leq &\int_M u^p \tilde G dv_\phi-2\int_M u^{p-1} e^\frac{w}{2}\nabla_\phi w\cdot_\phi \nabla_\phi \Delta_\phi w. 
\end{split}
\end{equation}
Using integration by parts to treat the last term in \eqref{gra01}, we have
\begin{equation}\label{gra02}
\int_{M}(p-1)u^{p-2}|\nabla_\phi u|_\phi^2 dv_\phi\leq 2p\int_M u^pGe^F dv_g,
\end{equation}
where $G=((n+\Delta\phi)^{2n-2}+1)$. The only difference is that we have only lower bound on $F$, uniform bound on $|F+f|$, but no upper bound of $F$, hence $e^F$ might not be bounded above a priori. We can then proceed as in \eqref{first15}-\eqref{first21} to get, with $\tilde L_\epsilon=\sqrt{\|Ge^F\|_{L^{\frac{1}{\epsilon}}}}$, 
\begin{equation}\label{gra03}
\|v^{\frac{p}{2}}\|_{L^{\frac{2(1-\epsilon)(2n+1)}{2n-1+2\epsilon}}}\leq C_{17} (p K_\epsilon \tilde L_\epsilon+1) \|v^{\frac{p}{2}}\|_{L^{\frac{2}{1-\epsilon}}}
\end{equation}
Taking $\epsilon=(2n+2)^{-1}$ such that 
\[
\frac{2(1-\epsilon)(2n+1)}{2n-1+2\epsilon}>\frac{2}{1-\epsilon}
\]
Note that $K_\epsilon=\|n+\Delta\phi\|_{L^{\frac{1}{\epsilon}-1}}<\infty$, hence we only need to bound $\tilde L_\epsilon$, then the iteration process gives 
\begin{equation}\label{gra04}
\|u\|_{L_\infty}\leq C_{18}\|u\|_{L^1}. 
\end{equation}
While $G=(n+\Delta\phi)^{2n-2}+1$ and $e^{F}\leq (n+\Delta\phi)^n$, hence $\tilde L_\epsilon<\infty$ by Lemma \ref{pp01}. The $L^1$ bound of $u$ follows exactly the same as in \eqref{first24}. This completes the proof. 
\end{proof}

\begin{cor}
With the same assumption as in Theorem \ref{main1}, we have for $1\leq p\leq p_0-1$,
\[\|\nabla (F+f)\|_{L^{2p}}\leq C\]
\end{cor}
\begin{proof} We get $L^{2p}$ bound of $|\nabla (F+f)|$ using the bound on  $|\nabla_\phi(F+f)|_\phi $ together with $L^p$ bound of $(n+\Delta\phi)$,
\[
|\nabla (F+f)|^2\leq |\nabla_\phi(F+f)|^2_\phi (n+\Delta \phi).
\] This completes the proof.  
\end{proof}

\section{Existence and properness}

In this section we prove Theorem \ref{cscs}, which we recall as follows,
\begin{thm}Let $(M, \xi, \eta, g)$ be a compact Sasaki manifold with fixed Reeb vector field $\xi$, fixed transverse holomorphic structure and fixed transverse K\"ahler class $[\omega^T]=[2^{-1}d\eta]$. Then there exists a cscs induced by transverse K\"ahler potentials in 
\[
\cH=\{\phi\in C^\infty_B(M): \omega_\phi=\omega^T+\sqrt{-1}\p_B\bar\p_B\phi>0\}
\]
if and only if the $\cK$-energy is (reduced) proper with respect to $d_{1, G}$ over $\cH$. 
\end{thm}

We  consider the continuity path of Chen \cite{chen15},  adapted to transverse K\"ahler structure in Sasaki setting, 
\begin{equation}\label{continuity1}
t(R_\phi-\underline R)-(1-t) (tr_\phi \omega^T-n)=0.
\end{equation}
We can assume that $\omega^T, \omega_\phi$ are both $K$-invariant.
The functional corresponding the continuity path \eqref{continuity1} is denoted by
\begin{equation}
\tilde K_{t}=t\cK(\phi)+(1-t)\mathbb{J}(\phi).
\end{equation}

\begin{rmk}Chen's continuity path \cite{chen15} plays a very central role in the proof. The insight is, for $t<1$ the solution corresponds to the minimizer of a strictly convex functional $\tilde K_t$, while $\cK$-energy is only convex. The strict convexity implies that the openness should be true (minimizer is unique, hence linearized operator should have zero kernel). More importantly, Chen's continuity path gives a canonical way to solve the csck and cscs equation, as an extension of Aubin's continuity path on Fano setting. \end{rmk}

\subsection{Properness implies existence}

The openness follows from the results of \cite{chen15, hashimoto, yuzeng} in K\"ahler setting. 
Technically the following theorem is an adaption of Hashimoto's result \cite{hashimoto}[Theorem 1.2] to Sasaki case. 
\begin{thm}\label{open1}On the compact Sasaki manifold $(M, \xi, \eta, g)$, suppose we have two transverse K\"ahler metrics $\omega^T$ and $\alpha$ such that $\Lambda_{\omega^T} (\alpha)=\text{const}. $ We assume $\omega^T, \alpha$ are both $K$-invariant.  Then there exists a constant $r(\omega^T, \alpha)$ depending only on $\omega^T, \alpha$ such that for $r\geq r(\omega^T, \alpha)$, there exists $\phi\in \cH_K$ satisfying $R_\phi-\Lambda_{\omega_\phi}(r\alpha)=\text{const}$.\end{thm}

The proof is a direct adaption of \cite{hashimoto}[Theorem 1.2] to the transverse K\"ahler structures on a Sasaki manifold, with functional spaces replaced by $K$-invariant functional spaces (in particular, all structures involved are invariant along the Reeb flow, hence are basic) and K\"ahler structures replaced by transverse K\"ahler structure, and the volume form replaced by the volume form $\omega^n_T\wedge \eta$. Otherwise, the arguments work almost word by word. Hence we skip the details. 
Similar to \cite{hashimoto}[Corollary 1.4], we also have the following,
\begin{cor}\label{open2}Suppose that $\omega^T, \alpha$ are two $K$-invariant transverse K\"ahler metrics such that $R_{\omega^T}-\Lambda_{\omega^T}\alpha=\text{const}$. Then if $\tilde \alpha$ is a $K$-invariant transverse K\"ahler metric such that $\tilde \alpha-\alpha$ is sufficiently small in the $C^\infty$-norm, then there exists a transverse K\"ahler metric $\tilde \omega\in [\omega^T]$ which is $K$-invariant, such that 
\[
R_{\tilde \omega}-\Lambda_{\tilde \omega}\tilde \alpha=\text{const}. 
\]  
\end{cor}

Theorem \ref{open1} and Corollary \ref{open2} imply that we can solve the equation \eqref{continuity1} in an interval $[0, t_0)$ for some $0<t_0\leq 1$. 
Now we show that $t_0=1$ under the assumption that the  $\cK$-energy is bounded below in $\cH_K$. 
Our argument is a modification of Chen-Cheng \cite{CC3}[Section 3]. 
Note that we can rewrite the equation as
\begin{equation}\label{continuity2}
\begin{split}
&\omega_\phi^n=e^F \omega^n\\
&\Delta_\phi F=-(\underline R-(1-t)n/t)+\text{tr}_\phi(Ric(\omega^T)-(1-t)\omega^T/t)
\end{split}
\end{equation}

Given the estimates above, we have a direct consequence that \eqref{continuity1} has a smooth solution for $t\in [0, 1)$. 
\begin{lemma}\label{kbelow1}Suppose the  $\cK$-energy  is bounded below on $\cH_K$, then \eqref{continuity1} has a smooth solution for $t\in [0, 1)$. 
\end{lemma}

\begin{proof}Since $\cK$ is bounded below on $\cH_K$, it follows that \[\tilde K_t\geq C_1(1-t)d_1(0, \phi)-C\] for any $t\in [0, 1)$. This depends on \eqref{ij3},
\[
\mathbb{J}(\phi)\geq  \frac{1}{n+1}I(\phi), 
\]
and a direct adaption of \cite{DR}[Proposition 5.5] to the Sasaki setting (together with Theorem \ref{pluri01})
\[
I(\phi)\geq C^{-1}d_1(0, \phi)-C, \phi\in \cH_0.
\]
Suppose $\phi(t)\in \cH^0$ solves \eqref{continuity1}. By the convexity of the $\cK$ and  the convexity of $\tilde K_t$ (along the $C^{1, \bar 1}$ geodesics in $\cH$),  it follows that $\phi(t)$ minimizes $\tilde K_t$ over $\cH$. Hence we have
$\tilde K_t<\infty$.    
It follows that for any $t\in [0, 1)$,
\[d_1(0, \phi)<C((1-t)^{-1}+1)\] 
Note that both $\mathbb{J}_{-Ric}(\phi)$ is bounded by $Cd_1(0, \phi)$ for some uniformly bounded constant $C$, as an application of Theorem \ref{pluri01}. It follows that, for any $t\in [0, 1)$, 
\begin{equation}\label{t0}
n!H(\phi)=\int_M \log \frac{\omega_{\phi}^n}{\omega^n_T} \omega_{\phi}^n\wedge \eta<C((1-t)^{-1}+1).
\end{equation}
It then follows from Theorem \ref{main0} that for any $t\in [0, 1)$, \eqref{continuity1} has a smooth solution. 
\end{proof}

Now we consider the behavior when $t\rightarrow 1$. The purpose is to show for any sequence $t_i\rightarrow 1$, after modifying by suitable choice of elements in $G$ (which preserves the Reeb vector field $\xi$ and transverse complex structure), there exists $ \phi_i\in \cH_K$ which induces a Sasaki metric through $\eta_{ \phi_i}$ and $\omega_{\phi_i}$, such that the limit of $\phi_i$ (by subsequence) defines a smooth cscs. Note that the estimate in \eqref{t0}  blows up when $t\rightarrow 1$. We need to use the reduced properness of $\cK$-energy in an effective way.  Since the properness only implies a distance bound of $d_{1, G}$, it is then necessary to apply an automorphism $\sigma_i$ in $G$ (at each time $t_i$) such that the resulting potential remains in a bounded set of $\cH_K$. We proceed as follows. 
Let $\tilde \phi_i\in \cH^0_K$ be the solution of \eqref{continuity1} at $t_i$, for $t_i$ increasing to $1$. The starting point is to show that $\tilde \phi_i$ is a minimizing sequence of the  $\cK$ energy. 
\begin{lemma}\label{kenergy1}We have the following, 
\begin{equation}\label{j100}
\begin{split}
&\tilde K_{t_i}(\tilde \phi_i)=\inf_{\phi\in \cH_K} \tilde K_{t_i}(\phi)\rightarrow \inf_{\cH_K} \cK(\phi), t_i\rightarrow 1\\
&\cK(\tilde \phi_i)\rightarrow \inf_{\cH_K} \cK(\phi), t_i\rightarrow 1\\
&(1-t_i)\mathbb{J}(\tilde \phi_i)\rightarrow 0.
\end{split}
\end{equation}
\end{lemma} 
\begin{proof}
Since $\cK$ is bounded below over $\cH_K$, we can choose $\phi^\epsilon \in\cH_K$ such that $\cK(\phi^\epsilon)\leq \inf_{\cH_K} \cK+\epsilon$.
Note that $t\cK+(1-t)\mathbb{J}$ is convex and hence $\tilde \phi_i$ minimizes $\tilde K_{t_i}$. 
Hence
\[ \tilde K_{t_i}(\tilde \phi_i)\leq t_i\cK(\phi^\epsilon)+(1-t_i)\mathbb{J}(\phi^\epsilon).
\]
It follows that 
\[
\lim\sup_{i\rightarrow\infty} \tilde K_{t_i}(\tilde \phi_i)\leq \cK(\phi^\epsilon)\leq \inf_{\cH_K} \cK+\epsilon.
\]
On the other hand, since $\mathbb{J}\geq 0$, for $i$ sufficiently large,
\[
t_i \inf_{\cH_K} \cK\leq t_i\cK(\tilde \phi_i)\leq \tilde K_{t_i}(\tilde\phi_i)\leq \inf_{\cH_K} \cK+\epsilon
\]
This proves all three statements in \eqref{j100}. 
\end{proof}
As a direct consequence of properness with respect to $d_{1, G}$ and Lemma \ref{kenergy1}, this gives the desired distance bound modulo $G$.
\begin{cor}We have the  bound on $d_{1, G}$, 
\[
\sup_i d_{1, G}(0, \tilde \phi_i)<\infty. 
\]
\end{cor}
Hence we can find a $\sigma_i\in G, \phi_i\in \cH^0_K$ such that
\begin{equation}
\eta_{\phi_i}=\sigma_i^*(\eta_{\tilde \phi_i});\; \omega_{\phi_i}=\sigma_i^*{\omega_{\tilde \phi_i}},\;\text{and}\; \sup_{i} d_1(0, \phi_i)<\infty. 
\end{equation}
With the uniform bound of distance $d_1(0, \phi_i)$, we need to show that $\phi_i$ converges uniformly. 
We show that the entropy of $\phi_i$ is uniformly bounded in the next,
\begin{prop}\label{e1000}
We have
\begin{equation}\label{entropy100}
H_\infty:=\sup_in!H(\phi_i)=\sup_i \int_M \log \frac{\omega_{\phi_i}^n}{\omega^n} \omega_{\phi_i}^n<\infty
\end{equation}
\end{prop} 
\begin{proof}
By Lemma \ref{fm02}, $\cK$ is invariant under the action of $G$. Hence  \[\sup_i\cK(\phi_i)=\sup_i \cK(\tilde \phi_i)<\infty.\]
Recall 
\[
\cK(\phi)=H(\phi)+\mathbb{J}_{-Ric}(\phi)
\]
Since  $|\mathbb{J}_{-Ric}|$ is uniformly  bounded when $\sup_id_1(0, \phi_i)<\infty$. This implies that $H_\infty<\infty.$ \end{proof}

We need to consider the equation which $\phi_i$ satisfies. Denote $\omega_i=\sigma_i^* (\omega_T)=\omega^T+\sqrt{-1}\p_B\bar\p_B h_i$, with the normalization $\sup h_i=0$ (note that $h_i$ is in $\cH_K$, but not in $\cH^0_K$ in general). 

\begin{lemma}\label{l100}The potential $\phi_i$ satisfies the following equations
\begin{equation}\label{equation-i0}
\begin{split}
&\omega_{\phi_i}^n\wedge \eta=e^{F_i}\omega^n_T\wedge \eta \\
&\Delta_{\phi_i} F_i=(\underline R-\frac{1-t_i}{t_i}n)+\text{tr}_{\phi_i}(Ric(\omega^T)-\frac{1-t_i}{t_i}\omega_i),
\end{split}
\end{equation}
\end{lemma}

\begin{proof}This is a direct adaption of computation in Chen-Cheng \cite{CC3}[Lemma 3.1]; we skip the details
\end{proof}

With the preparation above, we can then state the main theorem which gives a smooth cscs.
\begin{thm}\label{con100}When $t_i\rightarrow 1$, $\omega_{\phi_i}$ converges smoothly to a smooth transverse K\"ahler metric $\omega_{\phi}$ with constant scalar curvature. 
\end{thm}
\begin{proof}The argument proceeds almost identical to Chen-Cheng \cite{CC3}[Section 3], given Theorem \ref{main1}. 
Denote 
\[
R_i=\underline{R}-\frac{1-t_i}{t_i}n, \beta_i=\frac{1-t_i}{t_i}\omega_i, (\beta_0)_i=\frac{1-t_i}{t_i}\omega, f_i=\frac{1-t_i}{t_i}h_i.
\]
As in Chen-Cheng \cite{CC3}[Lemma 3.12], we prove that (using Tian's $\alpha$-invariant \cite{Tian97} and its adaption in Sasaki setting \cite{zhang}),  for any $p>1$, there exists $\epsilon_p>0$ such that if $t_i\in (1-\epsilon_p, 1)$, one has
\[
\int_M e^{-pf_i}\omega^n_T\wedge \eta=\int_M e^{-p\frac{1-t_i}{t_i}h_i}\omega^n_T\wedge \eta=\left(\int_M e^{-\alpha h_i}\omega^n_T\wedge \eta\right)^{p\frac{1-t_i}{\alpha t_i}} \text{Vol}(M)^{1-p\frac{1-t_i}{\alpha t_i}}\leq C
\]
for a constant $C$ uniformly bounded (independent of $p$ due to the choice of $\epsilon_p$). Hence Theorem \ref{main1} applies to get
the following estimate
\[
\|F_i+f_i\|_{W^{1, 2p}}+\|n+\Delta \phi_i\|_{L^p}\leq C_1,
\]
where $C_1=C_1(p, \omega^T, H_\infty)$ (see \eqref{entropy100}). By taking a subsequence, we can pass to the limit to get $K$-invariant functions $\phi_*\in W^{2, p}, F_{*}\in W^{1, p}$ for any $p<\infty$  such that
\begin{equation}
\begin{split}
&\phi_i\rightarrow \phi_*\; \text{in}\; C^{1, \alpha}\; \text{and}\; \sqrt{-1}\p\bar\p \phi_i\rightarrow \p\bar\p \phi_*\;\text{weakly in}\; L^p\\
&F_i+f_i\rightarrow F_{*}\;\text{in}\; C^{\alpha}\; \text{and}\; \nabla(F_i+f_i)\rightarrow \nabla F_*\;\text{weakly in}\; L^p\\
&\omega_{\phi_i}^n\wedge \eta\rightarrow \omega_{\phi_*}^n\wedge \eta\;\text{weakly in}\; L^p.
\end{split}
\end{equation}
It follows that $\phi_*$ is a weak solution of transverse csck in the following sense,
\begin{equation}\label{vol100}\omega^n_{\phi_*}\wedge \eta=e^{F_{*}}\omega^n_T\wedge \eta,\end{equation}
and for any $u\in C^\infty(M)$, we have
\begin{equation}\label{v102}
-\int_M d^c F_*\wedge du\wedge \frac{\omega^{n-1}_{\phi_*}}{(n-1)!}\wedge \eta=-\int_M u(\underline R)\frac{\omega_{\phi_*}^n}{n!}\wedge \eta+uRic\wedge \frac{\omega^{n-1}_{\phi_*}}{(n-1)!}\wedge \eta.
\end{equation}
Next we claim that $\int_M|f_i|\omega^n\rightarrow 0$. Given the claim, it follows that $e^{-f_i}\rightarrow 1$ in $L^p$ for any $p<\infty$ by a modified Lebesgue's dominated convergence theorem since $\sup_i\int_Me^{-p^{'}f_i}\omega^n<\infty$  (take $p<p^{'}$). 

By \eqref{j100},  $(1-t_i)\mathbb{J}(\tilde \phi_i)\rightarrow 0$ when $i\rightarrow \infty$. This implies that $(1-t_i)d_1(0, \tilde \phi_i)\rightarrow 0$. 
Note that we have the normalization condition $\sup h_i=0$. Denote $\tilde h_i=h_i-\mathbb{I}(h_i)\text{Vol}^{-1}(M)\in \cH_0$.  It follows that
\[
d_1(0, \tilde h_i)-d_1(0, \phi_i)\leq d_1(\tilde h_i,  \phi_i)=d_1(\sigma[0], \sigma[\tilde\phi_i])=d_1(0, \tilde \phi_i),
\]
in the last step we know that $G$ acts on $\cH_0$ isometrically. Hence we get $(1-t_i)d_1(0, \tilde h_i)\rightarrow 0$. By Theorem \ref{pluri01}, this implies that
$(1-t_i)\int_M |\tilde h_i|\omega^n\rightarrow 0.$ Since $\sup h_i=0$, we have for a uniformly bounded $C$, 
\[
0\leq \int_M (-h_i)\omega^n\leq C
\]
We have $\mathbb{I}(h_i)\leq 0$ and \[\mathbb{I}(h_i)-\int_M h_i\frac{\omega^n_T}{n!}\wedge \eta=\mathbb{J}(h_i)\geq 0\]
It follows that $\mathbb{I}(h_i)$ is uniformly bounded, and hence $\int_M|(1-t_i)h_i|\omega^n_T\wedge \eta\rightarrow 0$. This proves the claim $\int_M|f_i|\omega^n_T\wedge \eta\rightarrow 0$. This in particular proves \eqref{vol100}. By Theorem \ref{2nd}, we know that there exists a uniform constant $C_0>1$ such that
\begin{equation}\label{ch100}
C_0^{-1}\omega^T \leq \omega_{\phi_*}\leq C_0\omega^T, \phi_*\in W^{3, p}.
\end{equation}
By \eqref{equation-i0}, we have
\[\Delta_{\phi_i} (F_i+f_i)=(\underline R-\frac{1-t_i}{t_i}n)+\text{tr}_{\phi_i}(Ric(\omega^T)-\frac{1-t_i}{t_i}\omega^T)
\]
For any smooth function $u$, we write
\begin{equation*}
\int_M  (F_i+f_i) d^cdu\wedge \frac{\omega^{n-1}_{\phi_i}}{(n-1)!}\wedge \eta=-\int_M u(\underline R-\frac{1-t_i}{t_i}n)\frac{\omega_{\phi_i}^n}{n!}\wedge \eta+u(Ric-\frac{1-t_i}{t_i}\omega)\wedge \frac{\omega^{n-1}_{\phi_i}}{(n-1)!}\wedge \eta.
\end{equation*}
Given the convergence of $F_i+f_i\rightarrow F_*$ in $C^\alpha$, $\omega_i^k$ converges weakly and $\phi_i$ converges to $\phi_*$ in $C^\alpha$, we can then pass to the limit to get
\begin{equation}\label{v105}
\int_M  F_* d^cdu\wedge \frac{\omega^{n-1}_{\phi_*}}{(n-1)!}\wedge \eta=-\int_M u\underline R\frac{\omega_{\phi_i}^n}{n!}\wedge \eta+u Ric\wedge \frac{\omega^{n-1}_{\phi_*}}{(n-1)!}\wedge \eta.
\end{equation}
The standard elliptic theory together with \eqref{ch100} implies that $\phi_*$ defines a smooth transverse csck. 
\end{proof}

\subsection{Existence of cscs implies properness modulo $G$}\label{p1000}
The main result of this section is to prove the following regularity result, which generalizes Chen-Cheng's result \cite{CC2}[Theorem 5.1] to Sasaki case. 
\begin{thm}\label{regularity1}Let $\phi_*\in \cE_1$ be a minimizer of $\cK$ over $\cE^1$. Then $\phi_*$ is smooth and it defines a smooth cscs. 
\end{thm}

\begin{rmk}In K\"ahler setting, it is a conjecture of Darvas-Rubinstein \cite{DR} that a minimizer of $\cK$-energy in $\cE^1$ is a smooth csck and it was proved by Chen-Cheng \cite{CC2}. Earlier the author, together with  Y. Zeng \cite{HZ}, proved partly Chen's conjecture that a $C^{1, 1}$ minimizer of $\cK$ is a smooth csck.  Our argument is Sasaki counterpart of Chen-Cheng  \cite{CC2}[Theorem 5.1]. This result is a perfect example as a delicate combination of a priori estimates from PDE and pluripotential theory. 
\end{rmk}

\begin{proof}Let $\phi_*$ be a minimizer of $\cK$ over $\cE^1$. By Lemma \ref{approx13} below, we can find a sequence $\phi_i\in \cH$ such that $d_1(\phi_i, \phi_*)\rightarrow 0$ and also the entropy converges, $H(\phi_i)\rightarrow H(\phi_*)$. Since $\mathbb{J}_{-Ric}$ is $d_1$ continuous, it follows that $\cK(\phi_i)\rightarrow \cK(\phi_*)$. 
We consider the following continuity path, for each $i$, 
\begin{equation}\label{continuity3}
t(R_\phi-\underline{R})-(1-t)(\text{tr}_\phi \omega_{\phi_i}-n)=0. 
\end{equation}
The openness follows from Theorem \ref{open1} and Corollary \ref{open2} by taking $\alpha=\omega_{\phi_i}$. Denote the functionals $\mathbb{J}_i$, $\mathbb{I}_i, J_i, I_i$ to be the corresponding functionals with the base transverse K\"ahler form as $\omega_{\phi_i}$. 
By definition we have
\[\mathbb{J}_i(\phi)=\mathbb{J}(\omega_{\phi_i}, \omega_\phi), I_i(\phi)=I(\omega_{\phi_i}, \omega_\phi), J_i(\phi)=J(\omega_{\phi_i}, \omega_\phi)\]
The corresponding functional for the path \eqref{continuity3} is
\[
\tilde K_t^i:=t\cK+(1-t)\mathbb{J}_i.
\]
Since for each $i$, $\mathbb{J}_i$ is proper, hence arguing exactly as in  Lemma \ref{kbelow1}, we can get a unique smooth solution $\phi_i^t$ to the equation \eqref{continuity3}, for $t\in [0, 1)$ with $\mathbb{I}_\omega(\phi_i^t)=0$. Note that $\phi^t_i$ minimizes $\tilde K_t^i$ over $\cE_1.$  It follows that for any $\phi\in \cE_1$
\begin{equation}\label{min1}
t\cK(\phi_i^t)+(1-t_i)\mathbb{J}_i(\phi_i^t)\leq t \cK(\phi)+(1-t)\mathbb{J}_i(\phi)
\end{equation}
Taking $\phi=\phi_i$ in \eqref{min1} (noting that $\phi_i$ minimizes $\mathbb{J}_i(\phi)$), we have
\begin{equation}\label{k01}
0=\mathbb{J}_i(\phi_i)\leq \mathbb{J}_i(\phi_i^t), \;\text{and hence}\; \cK(\phi_i^t)\leq \cK(\phi_i)
\end{equation}
Taking $\phi=\phi_*$ in \eqref{min1} (noting that $\phi_*$ minimizes $\cK$), we have
\[
\mathbb{J}_i(\phi_i^t)\leq \mathbb{J}_i(\phi_*).
\]
In other words, we have
\begin{equation}
0=\mathbb{J}_i(\phi_i)\leq \mathbb{J}_i(\phi_i^t)\leq \mathbb{J}_i(\phi_*).
\end{equation}
Hence we have, by Theorem \ref{pluri01}
\[
\mathbb{J}_i(\phi_*)=(I_i-J_i)(\phi_*)\leq \frac{1}{n+1}I_i(\phi_*)\leq Cd_1(\phi_i, \phi_*)\rightarrow 0. 
\]
Note that we have,
\[
\mathbb{J}_i(\phi_i^t)=(I_i-J_i)(\phi_i^t)\geq  \frac{1}{n+1}I_i(\phi_i^t)
\]
Hence it follows that $I_i(\phi_i^t)\rightarrow 0$ when $i\rightarrow \infty$, uniform in $t$. By an adaption of \cite{BBEGZ}[Theorem 1.8] to Sasaki setting (see Lemma \ref{ifunctional}), we have
\[
I_\omega(\phi_i^t)\leq C_n(I_{\omega}(\phi_i)+I_i(\phi_i^t))\leq C
\]
This gives the distance bound  (see Theorem \ref{pluri01}), 
\[
d_1(0, \phi_i^t)\leq CI_{\omega}(\phi_i^t)+C.
\]
Together with the $\cK$-energy bound \eqref{k01}, this gives the uniform upper bound of the entropy $H(\phi_i^t)$, uniformly in $i$ and $t\in (0, 1)$. By Theorem \ref{main0}, we conclude that $\phi^t_i$ converges to a smooth $u_i$ when $t\rightarrow 1$, such that $u_i\in \cH_K$ solves the equation
\[
R_{u_i}-\underline R=0.
\]
In other words, $u_i$ defines a transverse K\"ahler metric for each $i$ and we have $I_i(u_i)\rightarrow 0$ when $i\rightarrow \infty$. Note that $d_1(0, u_i)$ is also uniformly bounded, Lemma \ref{compactness} implies that by passing to a subsequence if necessary, $u_i$ converges to a smooth  potential $u\in \cH_0$. It follows that
\[
I(\omega_u, \omega_{\phi_*})\leq C \left(I_i(u_i)+I(\omega_{\phi_i}, \omega_{\phi_*})\right)\rightarrow 0. 
\] 
This implies that $u$ differs by $\phi_*$ by a constant (\cite{CC2}[Lemma 5.7]). 
\end{proof}

We need an approximation in $d_1$ with convergent entropy in a $K$-invariant way.

\begin{lemma}\label{approx13}Given $u\in \cE_{1, K}$, there exists $u_k\in \cH_K$ such that $d_1(u, u_k)\rightarrow 0$ and $H_\omega(u_k)\rightarrow H_\omega(u)$, where the entropy is defined to be
$H_\omega(u)=\frac{1}{n!}\int_M \log(\frac{\omega_u^n}{\omega^n_T}) \omega^n_T\wedge \eta$. 
\end{lemma}

\begin{proof}We shall need the following results adapted to Sasaki setting: first $d_1$ convergence in $\cE_1$ implies the weak convergence of the complex Monge-Ampere measure, second uniqueness of complex Monge-Ampere equation in $\cE_1$ class. We shall establish these facts also in Sasaki setting in \cite{HL}.
Given these extensions to Sasaki setting,
this is a direct adaption of \cite{He18}[Lemma 3.7] and $K$ can be any compact subgroup of $\text{Aut}_0(\xi, J)$. And our argument is a modification of \cite{BDL}[Lemma 3.1]. We shall skip the details. \end{proof}

As a direct application, we have the following,
\begin{thm}Suppose $(M, \xi, \eta, [\omega^T], J)$ admits a cscs metric. Then the $\cK$ energy is $d_{1, G}$ proper over $\cH$, as defined in Definition \ref{rp} (not necessarily over $\cH_K$). In particular we get two constants $C, D>0$ such that
\[
\cK(\phi)\geq Cd_{1, G}(0, \phi)-D.
\]
\end{thm}
\begin{proof}We argue that $\cK$ is $d_{1, G}$ proper in $\cE_1$.  We choose the base metric $\omega^T$ to be a transverse csck. Suppose otherwise, there exists $\phi_i\in \cE_1$ such that $\cK(\phi_i)\leq C$, but $d_{1}(0, \sigma[\phi_i])\rightarrow \infty$ for any $\sigma\in G$.  Connecting $0$ and $ \phi_i$ by the finite energy geodesic with unit speed as in Theorem \ref{hl01} (see \cite{D4}[Theorem 3.36] also for K\"ahler setting). Consider the point $u_i$ along the geodesic such that $d_1(0, u_i)=1$. The convexity of $\cK$ along the finite energy geodesic (Lemma \ref{hl02}) implies that 
\begin{equation}\label{p101}
\frac{\cK(u_i)}{d_1(0, u_i)}\leq \frac{\cK(\phi_i)}{d_1(0, \phi_i)}\rightarrow 0. 
\end{equation}
Given $\cK(u_i)\rightarrow 0$ and $d_1(0, u_i)=1$, the compactness result in Sasaki setting (Lemma \ref{compactness}) implies that $u_i$ converges to $u$ in $d_1$-topology (by subsequence), and $\cK(u)=0$ (lower semicontinuity). It follows that $u$ is a smooth transverse csck by Theorem \ref{regularity1}. 

For $\tilde \phi_i\in \cE_1$ with $\tilde \phi_i=\sigma [\phi_i]$ for some $\sigma\in G$, such that $d_1(0, \tilde \phi_i)\rightarrow \infty$. Since $\cK$ is $G$-invariant, $\cK(\tilde \phi_i)\leq C$. The discussion above then applies to $\tilde \phi_i$. Connecting $0$ and $\tilde \phi_i$ by the geodesic with $\tilde u_i=\sigma[u_i]$, where $\tilde u_i$ is the point along the geodesic such that $d_1(0, \tilde u_i)=1$. By the discussion above, we get that $d_1(0, \tilde u)=1$ and $\tilde u=\sigma [u]$ such that $\cK(\tilde u)=0$.  Since this discussion holds for any $\sigma\in G$, this implies that $d_{1, G}(0, u)=1$. This contradicts the fact that $G$ acts transitively on cscs metrics \cite{JZ, VC}. By checking the proof more carefully,  when $d(0, \phi_i)\rightarrow \infty$, if we have $\frac{\cK(\phi_i)}{d_1(0, \phi_i)}\rightarrow 0,$
then \eqref{p101} still applies to get that $\cK(u)=0$, and it leads to contradiction. Hence, we can get constants $C, D>0$ such that, 
\[
\cK(\phi)\geq Cd_{1, G}(0, \phi)-D.
\]
\end{proof}

\section{Appendix}

We give a very brief discussion of how techniques used in K\"ahler geometry (see \cite{CC3} and \cite{D4} for example) extends to Sasaki setting for results in the paper. There are two K\"ahler structures which are essentially equivalent descriptions of Sasaki structure, the K\"ahler cone structure and the transverse K\"ahler structure. When the Reeb vector field is regular of quasi-regular, there is a global quotient $M/\cF_\xi$ which is a compact K\"ahler manifold (or orbifold). It is very straightforward to extend results in K\"ahler setting as in \cite{CC3, D4} to regular or quasi-regular Sasaki manifolds. When the Reeb vector field is irregular, it becomes much more complicated at times. Some results which are involved with only transverse K\"ahler structure (basic quantities) and are of global nature can be extended in a rather straightforward way; examples include \cite{CC3}[Theorem 2.1, Theorem 2.2] and \cite{chenhe12}[Theorem 1.1]. Some other results need to be dealt with both globally and locally, such as Chen-Cheng's $C^0$ estimates of $\phi$ and $F$, which we extend to Theorem \ref{firstorder}. For the pluripotential theory, including the proof of Theorem \ref{pluri01}, we follow the K\"ahler setting, see T. Darvas's very nice lecture notes \cite{D4}. We shall emphasize that the proofs for the K\"ahler setting are indeed very long and very intricate, and of course the K\"ahler setting contains the most essential ideas.  On the other hand,  new delicate difficulties do appear in the irregular case in Sasaki setting and it is quite tricky and lengthy to extend these results. Nevertheless, we are able to utilize the transverse  K\"ahler structure and K\"ahler cone structure in an effective way to extend almost all related results in \cite{D4} to the Sasaki setting. We shall present the details of these results in \cite{HL}, including the proof of Theorem \ref{pluri01}.

\end{document}